\documentclass[11pt]{article}

\usepackage[left=2.7cm,bottom=2.7cm,right=2.7cm,top=2.7cm]{geometry}

\usepackage{color,hyperref}

\usepackage[numbers]{natbib}
\newif\ifworkinprogress
\workinprogresstrue

\ifworkinprogress
\newcommand{\RM}[1]{\textcolor{blue}{\textbf{[RM] #1}}}
\newcommand{\HW}[1]{\textcolor{green}{\textbf{[HW] #1}}} 		  
\else
\newcommand{\RM}[1]{}
\newcommand{\HW}[1]{} 	
\fi

\usepackage{csquotes}
\usepackage{verbatim}

\usepackage{epsfig}
\usepackage{color}

\usepackage{balance}  
\usepackage{algorithmic}
\usepackage{algorithm}
\usepackage{multirow}
\usepackage{graphicx}
\usepackage{amsmath}
\usepackage{amssymb}
\usepackage{amsfonts}
\usepackage{xspace}
\usepackage{url}
\usepackage{bbm}
\usepackage{stmaryrd}

\newcommand{\R}{\mathbb{R}}
\newcommand{\T}{\top}
\newcommand{\la}{\langle}
\newcommand{\ra}{\rangle}

\newcommand{\lam}{\lambda}
\newcommand{\dt}{\delta}
\newcommand{\Dt}{\Delta}

\newcommand{\be}{\beta}
\newcommand{\ga}{\gamma}
\newcommand{\ep}{\epsilon}

\newcommand{\lt}{\left}
\newcommand{\rt}{\right}

\newcommand{\argmin}{\mathop{{\rm argmin}}}

\newcommand{\wtd}{\widetilde}

\newcommand{\Pb}{\mathbb{P}}

\newcommand{\Ex}{\mathbb{E}}

\newcommand{\dist}{\mathsf{dist}}
\newcommand{\supp}{\mathsf{supp}}

\newcommand{\itl}{\interleave}
\newcommand{\subG}{\mathsf{subG}}

\newcommand{\cB}{\mathcal{B}}

\newcommand{\cN}{\mathcal{N}}

\newcommand{\cS}{\mathcal{S}}

\newcommand{\cZ}{\mathcal{Z}}

\newcommand{\B}{\boldsymbol}
\newcommand{\tdH}{\widetilde{\boldsymbol{H}}}
\newcommand{\tdP}{\widetilde{\boldsymbol{P}}}
\newcommand{\hatP}{\hat{\boldsymbol{P}}}
\newcommand{\hatbeta}{\hat{\boldsymbol{\beta}}}

\newtheorem{theorem}{Theorem}[section]
\newtheorem{corollary}[theorem]{Corollary}
\newtheorem{lemma}[theorem]{Lemma}
\newtheorem{proposition}[theorem]{Proposition}

\newtheorem{assumption}[theorem]{Assumption}

\makeatletter
\@addtoreset{equation}{section}
\makeatother
\renewcommand{\theequation}{\arabic{section}.\arabic{equation}}



\begin{document}
	
	\title{Linear regression with partially mismatched data: \\ local search with theoretical guarantees
		\footnote{This research was supported in part, by grants from the Office of Naval Research: ONR-N000141812298
			(YIP), N000142112841, the National Science Foundation: NSF-IIS-1718258, IBM and Liberty Mutual Insurance, awarded to Rahul Mazumder.}
	}
\author{
	Rahul Mazumder\thanks{MIT Sloan School of Management, Operations Research Center and MIT Center for Statistics ({email: rahulmaz@mit.edu}).}
		\and
		Haoyue Wang\thanks{MIT Operations Research Center (email: haoyuew@mit.edu).}
}

\date{}
\maketitle
	
	\begin{abstract}
		Linear regression is a fundamental modeling tool in statistics and related fields. In this paper, we study an important variant of linear regression in which the predictor-response pairs are partially mismatched. 
		We use an optimization formulation to simultaneously learn the underlying regression coefficients and the permutation corresponding to the mismatches. The combinatorial structure of the problem leads to computational challenges. We propose and study a simple  greedy local search algorithm for this optimization problem that enjoys strong theoretical guarantees and appealing computational performance. 
		We prove that under a suitable scaling of the number of mismatched pairs compared to the number of samples and features, and certain assumptions on problem data; 
		our local search algorithm converges to a nearly-optimal solution at a linear rate.
		In particular, in the noiseless case, 
		our algorithm converges to the global optimal solution with a linear convergence rate. 
			Based on this result, 
			we prove an upper bound for the estimation error of the parameter. 
		We also propose an approximate local search step that allows us to scale our approach to much larger instances. We conduct numerical experiments to gather further insights into our theoretical results, and show promising performance gains compared to existing approaches. 
	\end{abstract}

\section{Introduction}
Linear regression and its extensions are among the most fundamental models in statistics and related fields. In the classical and most common setting, we are given $n$ samples with features $\B{x}_{i} \in \R^{d}$ and response $y_{i} \in \R$, where $i$ denotes the sample indices. We assume that the features and responses are perfectly matched i.e., $\B{x}_i$ and $y_i$ correspond to the same record or sample. 
However, in important applications (for example, due to errors in the data merging process), the correspondence between the response and features may be \emph{broken}~\cite{neter1965effect,pananjady2017denoising,slawski2020two}.
This erroneous correspondence needs to be adjusted before performing downstream statistical analysis.
Thus motivated, we consider a mismatched linear model with responses $\B{y} = [y_1,...,y_n]^\T\in \R^n$ and covariates $\B{X}=[\B{x}_1,..., \B{x}_n]^\T \in \R^{n\times d} $ satisfying 
\begin{eqnarray}\label{model1}
\B{P}^* \B{y} = \B{X} \B{\beta}^* + \B{\ep} 
\end{eqnarray}
where $\B{\beta}^*\in \R^d$ are the true regression coefficients, $\B{\ep} = [\ep_1,...,\ep_n]^\T \in \R^n$ is the noise term, and $\B P^*\in \R^{n\times n}$ is an unknown permutation matrix. 
We consider the classical setting where $n > d$ and $\B X$ has full rank; and seek to estimate both $\B \be^*$ and $\B P^*$ based on the $n$ observations $\{(y_{i}, \B x_{i})\}_{1}^{n}$. Note that the main computational difficulty in this task arises from the unknown permutation. 

Linear regression with mismatched/permuted data---for example, model \eqref{model1}---has a long history in statistics dating back to 1960s~\cite{neter1965effect,degroot1971matchmaking,degroot1976matching}. In addition to the aforementioned application in record linkage, similar problems also appear in robotics \cite{stachniss2016simultaneous}, multi-target tracking \cite{blackman1986multiple} and signal processing \cite{balakrishnan1962problem}, among others.  
Recently, this problem has garnered significant attention from the statistics and machine learning communities. A series of recent works \cite{emiya2014compressed,unnikrishnan2018unlabeled,pananjady2017denoising,pananjady2017linear,abid2017linear,abid2018stochastic,hsu2017linear,haghighatshoar2017signal,shi2020spherical,wang2018signal,dokmanic2019permutations,tsakiris2020algebraic,slawski2020two,slawski2021pseudo,slawski2020sparse} have studied the statistical and computational aspects of this model. To learn the coefficients $\B \be^*$ and the matrix $\B P^*$, one can consider the following natural optimization problem:
\begin{eqnarray}\label{intro: problem0}
\min_{\B \be, \B P}~ \| \B P \B y - \B X \B \be \|^2 ~~~ {\rm s.t.} ~~~ \B P \in \Pi_n
\end{eqnarray}
where $\Pi_n$ is the set of $n \times n$ permutation matrices.
Solving problem \eqref{intro: problem0} is difficult as there are exponentially many choices for $\B P \in \Pi_{n}$. Given $\B P$ however, it is easy to estimate $\B \beta$ via least squares.
\cite{unnikrishnan2018unlabeled} shows that in the noiseless setting ($\B \ep=\B 0$),
a solution $(\hat{\B{P}}, \hat{\B \be})$ of problem \eqref{intro: problem0} equals $( \B P^*, \B  \be^*)$ with probability one if $n\ge 2d$ and the entries of $\B X$ are independent and identically distributed (iid) as per a distribution that is absolutely continuous with respect to the Lebesgue measure.
\cite{pananjady2017linear,hsu2017linear} studies the estimation of $(\B P^*, \B \be^*) $ under the noisy setting. 
It is shown in
\cite{pananjady2017linear} that Problem \eqref{intro: problem0} is NP-hard if $d\ge \kappa n $ for some constant $\kappa>0$. A polynomial-time approximation algorithm appears in \cite{hsu2017linear} for a fixed $d$. 
However, as noted in~\cite{hsu2017linear}, this algorithm does not appear to be  efficient in practice. \cite{emiya2014compressed} propose a branch-and-bound method, that can solve small problems with $n\le 20$ (within a reasonable time). \cite{peng2020linear} propose a branch-and-bound method for a concave minimization formulation, which can solve problem~\eqref{intro: problem0} with $d\le 8$ and $n \approx 100$ (the authors report a runtime of 40 minutes to solve instances with $d=8$ and $n=100$). \cite{tsakiris2020algebraic} propose an approach using tools from algebraic geometry, which can handle problems with $d\le 6$ and $n = 10^3 \sim 10^5$---the cost of this method increases exponentially with $d$. This approach is exact for the noiseless case but approximate for the noisy case ($\B\ep \neq \B{0}$). 
Several heuristics have been proposed for~\eqref{intro: problem0}: Examples include, alternating minimization~\cite{haghighatshoar2017signal,wang2018signal},  Expectation Maximization~\cite{abid2018stochastic}---as far as we can tell, these methods are sensitive to initialization, and have limited theoretical guarantees.

As discussed in \cite{slawski2019linear,slawski2020two}, in several applications, a small fraction of the samples are mismatched --- that is, the permutation $\B{P}^*$ is \emph{sparse}.
In other words, 
 if we let $r:= |\{i\in [n] ~|~ \B P^* \B e_i  \neq  \B e_i \}|$ where $\B e_1, ..., \B e_n$ are the standard basis elements of $\R^n$,
then $r$ is much smaller than $n$. In this paper, we focus on such sparse permutation matrices, and assume the value of $r$ is known or a good estimate is available to the practitioner.
This motivates a constrained version of~\eqref{intro: problem0}, given by
\begin{equation}\label{intro: problem1}
\min_{\B \be, \B P}~ \| \B P \B y -\B  X \B \be \|^2 ~~~ {\rm s.t.} ~~~ \B P \in \Pi_n, ~ \dist (\B P,  \B I_n) \le R
\end{equation}
where, the constraint $\dist(\B P, \B{I}_{n}) \leq R$  restricts the number of mismatches between $\B P$ and the identity permutation to be below $R$---See~\eqref{permutation-dist-defn-1} for a formal definition of $\dist(\cdot,\cdot)$. Above, $R$ is taken such that $r \le R \le n$ (Further details on the choice of $R$ can be found in Sections \ref{section: theoretical guarantees} and \ref{sec:expts}). Note that as long as $r \le R \le n$, the true parameters $(\B P^*, \B \beta^*)$ lead to a feasible solution to \eqref{intro: problem1}. In the special case when $R = n$, the constraint $\dist(\B P, \B I_{n}) \leq R$ is redundant, and Problem \eqref{intro: problem1} is equivalent to problem \eqref{intro: problem0}. 
Interesting convex optimization approaches based on robust regression have been proposed in~\cite{slawski2019linear} to approximately solve~\eqref{intro: problem1} when $r \ll n$. The authors focus on obtaining an estimate of $\B\beta^*$.
Similar ideas have been extended to consider problems with multiple responses in~\cite{slawski2020two}.

Problem \eqref{intro: problem1} can be formulated as a mixed-integer program (MIP)
with $O(n^2)$ binary variables (to model the unknown permutation matrix). Solving this MIP with off-the-shelf MIP solvers (e.g., Gurobi) becomes computationally expensive for even a small value of $n$ (e.g. $n \approx 50$). 
To the best of our knowledge, we are not aware of computationally practical algorithms with theoretical guarantees that can
optimally solve the original problem~\eqref{intro: problem1}, under suitable assumptions on the problem data. Addressing this gap is the main focus of this paper: We propose and study a novel greedy local search method\footnote{We draw inspiration from the local search method presented in~\cite{hazimeh2020fast} in the context of a different problem: high dimensional sparse regression.} for Problem~\eqref{intro: problem1}.
Loosely speaking, our algorithm at every step performs a greedy \emph{swap} or transposition, in an attempt to improve the cost function. 
This algorithm is typically efficient in practice based on our numerical experiments. We also propose an approximate version of the greedy swap procedure that scales to much larger problem instances.
We establish theoretical guarantees on the convergence of the proposed method under suitable assumptions on the problem data.
Under a suitable scaling of the number of mismatched pairs compared to the number of samples and features, and certain assumptions on the covariates and noise; our local search method converges to an objective value that is at most a constant multiple of the squared norm of the underlying noise term. From a statistical viewpoint, this is the best objective value that one can hope to obtain (due to the noise in the problem).
Interestingly, in the special case of $\B \ep = \B 0$ (i.e., the noiseless setting),
our algorithm
converges to an optimal solution of~\eqref{intro: problem1} with a linear rate\footnote{The extended abstract \cite{mazumder2021linear} which is a shorter version of this manuscript, studies the noiseless setting.}.
	We also prove an upper bound of the estimation error of $\B \beta^*$ (in $\ell_2$ norm) and derive a bound on the number of iterations taken by our proposed local search method to find a solution with this estimation error.

\medskip

\noindent{\bf{Notation and preliminaries}:}
For a vector $\B a$, we let $\|\B a\|$ denote the Euclidean norm, $\| \B a \|_{\infty}$ the $\ell_{\infty}$-norm and $\| \B a\|_0$ the $\ell_0$-pseudo-norm (i.e., number of nonzeros) of $\B a$.	
We let $ \interleave \cdot \itl_2$ denote the operator norm for matrices. 
Let $\{\B e_1,...,\B e_n\}$ be the natural orthogonal basis of $\R^n$. 
For a finite set $S$, we let $ \# S $ denote its cardinality. For any permutation matrix $\B P$, let $\pi_{\B P}$ be the corresponding permutation of $\{1,2,....,n\}$, that is, $\pi_{\B P}(i) = j$ if and only if $\B e_i^\T \B P  = \B e_j^\T$ if and only if $P_{ij} = 1$. 
We define the distance between two permutation matrices $\B P$ and $\B Q$ as
\begin{equation}\label{permutation-dist-defn-1}
\dist(\B P,\B Q)  = \# \lt\{ i \in [n] : \pi_{\B P} (i) \neq \pi_{\B Q}(i)  \rt\}.  
\end{equation}
Recall that we assume $r = \dist(\B P^*, \B I_n)$. 
For a given permutation matrix $\B P$, define 
the $m$-neighbourhood of $\B P$ as 
\begin{equation}
\cN_m(\B P) := \{  \B Q \in \Pi_n  : ~ \dist(\B P,\B Q) \le m  \} .
\end{equation}
It is easy to check that $\cN_1(\B P) = \{\B P\}$, and for any $R\ge 2$, $\cN_R(\B P)$ contains more than one element. 
For any permutation matrix $\B P\in \Pi_n$, we define its support as:
$\supp(\B P) :=  \lt\{   i\in [n] :~ \pi_{\B P}(i) \neq i    \rt\}.$
For a real symmetric matrix $\B A$, let $\lam_{\max}(\B A)$ and $\lam_{\min }(\B A)$ denote the largest and smallest eigenvalues of $\B A$, respectively. 

For two positive scalar sequences $\{a_n\}, \{b_n\}$, we write 
$a_n =  O(b_n)$ or equivalently, $a_n/b_n =  O(1)$, if there exists a universal constant $C$ such that $ a_n \le C b_n $. We write 
$a_n = \Omega (b_n)$ or equivalently, $a_n/b_n =  \Omega(1)$, if there exists a universal constant $c$ such that $ a_n \ge c b_n $. 
We write $a_n = \Theta (b_n)$ if both $a_n = O(b_n)  $ and $a_n = \Omega(b_n)$ hold.

\section{A local search method}
\label{sec:local-search-sec}
Here we present our local search method for~\eqref{intro: problem1}. For any fixed $\B P\in \Pi_n$, by minimizing the objective function in \eqref{intro: problem1} with respect to $\B \beta$, we have an equivalent formulation
\begin{equation}\label{problem2}
\min_{\B P} ~~ \| \B P \B y  - \B H \B P \B y \|^2 ~~{\rm s.t.} ~~\B P\in \Pi_n , ~ \dist(\B P, \B  I_n) \le R \ ,
\end{equation}
where $\B H = \B X (\B X^\T \B X)^{-1} \B X^\T$ is the projection matrix onto the columns of $\B X$. To simplify notation, denote $ \wtd {\B H} := \B I_n - \B H$, then~\eqref{problem2} is equivalent to
\begin{equation}\label{problem3}
\min_{\B P} ~~ \| \wtd{\B H} \B P \B y \|^2 ~~ {\rm s.t.} ~\B P\in \Pi_n   , ~ \dist(\B P, \B I_n) \le R \ .
\end{equation}
Our local search approach for the optimization of Problem~\eqref{problem3} is summarized in Algorithm~\ref{alg: AFW}.
\begin{algorithm}[H]
	\caption{A local search method for Problem~\eqref{problem3}.}
	\label{alg: AFW}
	\begin{algorithmic}
		\STATE \textbf{Input}: Initial permutation $\B P^{(0)} = \B I_n$. Tolerance $\mathfrak{e} \geq 0$ and maximum number of iterations, $K$.
		
		\STATE For $k  = 0,1,2,....$ 
		\begin{eqnarray}\label{update}
		\B P^{(k+1)} \in \argmin_{\B P} \lt\{     \| \wtd{\B H} \B P\B y\|^2 :~  \dist(\B P, \B P^{(k)}) \le 2, ~ \dist(\B P, \B I_n )\le R \rt\}  .
		\end{eqnarray}

		~~~~
		If $  \| \wtd{\B H} \B P^{(k)} \B y \|^2 - \| \wtd{\B H} \B P^{(k+1)} \B y \|^2 \leq \mathfrak{e}$ or $k = K$, output $\B P^{(k)}$. 
	\end{algorithmic}
\end{algorithm}
At iteration $k$, Algorithm~\ref{alg: AFW} finds a swap  (within a distance of $R$ from $\B I_n$) that leads to the smallest objective value. 
To see the computational cost of~\eqref{update},  note that:
\begin{eqnarray}\label{eq1}
&&\| \wtd{\B H} \B P\B y \|^2 = \| \tdH (\B P-\B P^{(k)})\B y +\tdH \B P^{(k)}\B y  \|^2 \nonumber\\
&=&
\| \tdH (\B P-\B P^{(k)})\B y \|^2
+2 \la (\B P-\B P^{(k)})\B y,  \tdH \B P^{(k)} \B y   \ra  +  \|\tdH \B P^{(k)}\B y  \|^2 \ .
\end{eqnarray}
For each $\B P$,
with $\dist (\B P, \B P^{(k)}) \le 2$, the vector $(\B P-\B P^{(k)}) \B y$ has at most two nonzero entries. 
Since we pre-compute $\tdH$, 
computing the first term in \eqref{eq1} costs $O(1)$ operations. 
As we retain a copy of $\tdH \B P^{(k)}\B y $ in memory, computing the second term in~\eqref{eq1} also costs $O(1)$ operations. Therefore, computing~\eqref{update}
requires $O(n^2)$ operations, as there are at most $n^2$-many possible swaps to search over. The $O(n^2)$ per-iteration cost is quite reasonable for medium-sized examples with $n$ being a few hundred to a few thousand, but might be expensive for larger examples. In Section \ref{section: fast local search steps}, we propose a fast method to find an approximate solution of \eqref{update} that scales to instances with $n \approx 10^7$ in a few minutes (see Section~\ref{sec:expts} for numerical findings). 

\section{Theoretical guarantees for Algorithm~\ref{alg: AFW}}\label{section: theoretical guarantees}

Here we present theoretical guarantees for Algorithm~\ref{alg: AFW}. The main assumptions and conclusions appear in Section~\ref{subsection: Main results}. Section~\ref{subsection: proofs of main theorems} presents the proofs of the main theorems. The development in Sections~\ref{subsection: Main results} and~\ref{subsection: proofs of main theorems} assumes that the problem data (i.e., $\B y, \B X, \B \ep$) is deterministic.  
Section~\ref{subsection: sufficient conditions for assumptions} discusses conditions on the distribution of the features and the noise term, under which the main assumptions hold true with high probability.

\subsection{Main results}\label{subsection: Main results}

We state and prove the main theorems on the convergence of Algorithm \ref{alg: AFW}.  For any $m\le n$, define 
\begin{equation}\label{defn-Bm}
\cB_m:= \lt\{    \B w \in \R^n:~    \| \B w\|_0 \le m \rt\}.
\end{equation}
We first state the assumptions useful for our technical analysis.

\begin{assumption}\label{ass-main0}
		Suppose $\B X$, $\B y$, $\B \ep$, $\B\beta^*$ and $\B P^*$ satisfy the model \eqref{model1} with $\dist(\B P^*, \B I_n) \le r$. Suppose the following conditions hold:

	\noindent
	(1) There exist constants $U>L>0$ such that 
	\begin{equation}
	\max_{i,j\in[n]} |y_i-y_j| \le U,  ~~~ {\rm and } ~~~ | (P^*y)_i - y_i | \ge L   ~~~~~ \forall i\in \supp(\B P^*) \ . \nonumber
	\end{equation}
	(2) Set $R=10 C_1rU^2/L^2+4$ for some constant $ C_1 > 1$.\\
	(3) There is a constant $\rho_n =  O(d \log (n)/n)$ 
	such that $R \rho_n \le L^2/(90U^2)$, and 
	\begin{eqnarray}\label{ass-re}
	\| \B H \B u\|^2 \le \rho_n \| \B u\|^2 ~~ \forall \B u\in \cB_{4}, ~~\text{and} ~~\| \B H \B u\|^2 \le R\rho_n \| \B u\|^2 ~~ \forall \B u\in \cB_{2R} \ . 
	\end{eqnarray}
	(4) There is a constant $ \bar \sigma\ge 0$ 
	satisfying $\bar \sigma \le \min\{0.5, (\rho_nd)^{-1/2}\} L^2/(80U) $ 
	such that 
	\begin{equation}\label{ineq-ass1}
	\| \tdH \B \ep\|_{\infty} \le \bar\sigma, ~~ 
	\|\B \ep\|_\infty \le \bar \sigma, ~~\text{and}~~ \| \B H\B \ep \| \le \sqrt{d} \bar \sigma \ .  
	\end{equation}
\end{assumption}

Note that the lower bound in 
Assumption~\ref{ass-main0}~(1) states that the $y$-value for a record that has been mismatched is not too close to its original value (before mismatch).
Assumption~\ref{ass-main0}~(2) states that $R$ is set to a constant multiple of $r$. This constant can be large ($\ge 10U^2/L^2$), and appears to be an artifact of our proof techniques. Our numerical experience appears to suggest that this constant can be much smaller in practice. 
Assumption \ref{ass-main0}~(3) is a restricted eigenvalue (RE)-type condition~\cite{wainwright2019high} stating that:
a multiplication of any $(2R)$-sparse vector by $\B H$ will result in 
a vector with small norm (in the case $R\rho_n <1$). 
Section~\ref{subsection: sufficient conditions for assumptions} discusses conditions on the distribution of the rows of $\B X$ under which   Assumption \ref{ass-main0}~(3) holds true with high probability.
Note that if $\rho_n =  \Theta(d \log (n)/n)$, then for the assumption $R \rho_n \le L^2/(90U^2)$ to hold true, we require $n/\log (n) = \Omega(dr)$. 
Assumption \ref{ass-main0}~(4) limits the amount of noise $\B \ep$ in the problem. Section~\ref{subsection: sufficient conditions for assumptions} presents conditions on the distributions of $\B\ep$ and $\B X$ (in a random design setting) which ensures Assumption~\ref{ass-main0}~(4) holds true with high probability.

Assumption~\ref{ass-main0}~(3) plays an important role in our technical analysis. In particular, this allows us to approximate the objective function in~\eqref{problem3} with one that is easier to analyze. To provide some intuition,  we write
$\tdH \B P^{(k)}\B y =  \tdH(\B P^{(k)} \B y - \B P^* \B y) + \tdH \B \ep$---noting that $\tdH \B P^* \B y = \tdH (\B X\B \beta^* + \B \ep) = \tdH \B \ep$, and assuming that the noise $\B \epsilon$ is small, we have:
\begin{equation}\label{approx-obj}
\begin{aligned}
\| \tdH \B P^{(k)}\B y\|^2 \approx \| \tdH( \B P^{(k)} \B y - \B P^* \B y) \|
\approx  \| \B P^{(k)} \B y - \B P^* \B y \|^2. 
\end{aligned}
\end{equation}
Intuitively, the term on the right-hand side is the approximate objective that we analyze in our theory. Lemma~\ref{lemma: one-step} presents a one-step decrease property on the approximate objective function.  
\begin{lemma}\label{lemma: one-step}
	(One-step decrease)
	Given any $\B y\in \R^n $ and $\B P, \B P^* \in \Pi_n$, there exists a permutation matrix $\tdP \in \Pi_n$ such that $\dist(\tdP, \B P) = 2$, $\supp(\tdP  (\B P^*)^{-1}) \subseteq \supp( \B P  (\B P^*)^{-1}) $ and 
	\begin{equation}\label{one-step ineq1}
	\| \B P \B y - \B P^* \B y\|^2  - \| \tdP \B y - \B P^* \B y \|^2 \ge (1/2)  \| \B P \B y - \B P^* \B y\|^2_{\infty}  \ . 
	\end{equation}
	If in addition $ \| \B P \B y - \B P^* \B y \|_{0} \le m$ for some $m\le n$, then 
	\begin{equation}\label{one-step ineq2}
	\| \tdP \B y - \B P^* \B y\|^2  \le \lt( 1- 1/(2m)\rt)  \| \B P \B y - \B P^* \B y\|^2 \ . 
	\end{equation}
\end{lemma}

The main results 
make use of Lemma~\ref{lemma: one-step} and formalize the intuition conveyed in~\eqref{approx-obj}. 
We first present
a result regarding the support of the permutation matrix $\B P^{(k)}$ delivered by Algorithm~\ref{alg: AFW}.
\begin{proposition}\label{prop: support-inclusion}
	(Support detection)
	Suppose Assumption \ref{ass-main0} holds. Let $\{\B P^{(k)}\}_{k \geq 0}$ be the permutation matrices generated by Algorithm \ref{alg: AFW}. Then for all $k\ge R/2$, it holds $ \supp(\B P^{*}) \subseteq  \supp(\B P^{(k)})  $.
\end{proposition}

Proposition~\ref{prop: support-inclusion} states that the support of $\B P^*$ will be contained within the support of $\B P^{(k)}$ after at most $R/2$ iterations. 
Intuitively, this result is because of Assumption \ref{ass-main0} (1), which assumes that the mismatches represented by $\B P^*$ have ``strong signal". 
Proposition \ref{prop: support-inclusion} is also useful for the proofs of the main theorems below 
(e.g., see Claim \ref{claim1} in the proof of Theorem \ref{theorem: main result} for details). 

We now present some additional assumptions required for the results that follow.  
\begin{assumption}\label{ass-main}
	Let $\rho_n$ and $\bar\sigma$ be parameters appearing in Assumption \ref{ass-main0}.
	\\
	(1) Suppose $R^2 \rho_n \le 1/10$. \\
	(2) There is a constant $\sigma\ge 0$ such that $\bar \sigma^2  \le \sigma^2 \min \{n/(660R^2),$ $ n/(5dR)\}$, and 
	\begin{equation}\label{ineq-ass2}
	\|\tdH \B \ep \|^2 \ge (1/2)n \sigma^2.
	\end{equation}
\end{assumption}
In light of the discussion following Assumption~\ref{ass-main0}, Assumption \ref{ass-main} (1) places a stricter condition on the size of $n$ via the requirement $R^2 \rho_n \le 1/10$. If $\rho_n =  \Theta(d\log (n)/n)$, then we would need $n/\log(n) =\Omega( dr^2)$, which is stronger than the condition $n/\log(n) = \Omega(dr)$ needed in 
Assumption~\ref{ass-main0}. 

Assumption~\ref{ass-main}~(2) imposes a lower bound on $\|\tdH \B\ep\|$ -- this can be equivalently viewed as an upper bound on $\| \B H \B\ep \| $, in addition to the upper bound appearing in Assumption~\ref{ass-main0}~(4). 
Section~\ref{subsection: sufficient conditions for assumptions} provides a sufficient condition for Assumption~\ref{ass-main}~(2) to hold with high probability. In particular, in the noiseless case ($\B\ep = 0$),  Assumption~\ref{ass-main0}~(4) and  
Assumption~\ref{ass-main}~(2) hold with $\bar\sigma = \sigma = 0$.

We now state the first convergence result.

\begin{theorem}\label{theorem: main result}
	(Linear convergence of objective up to noise level)
	Suppose Assumptions \ref{ass-main0} and \ref{ass-main} hold with $R$ being an even number. Let $\{\B P^{(k)}\}_{k \geq 0}$ be the permutation matrices generated by Algorithm \ref{alg: AFW}. Then 
	for any $k\ge 0$, we have 
	\begin{equation}\label{upper-bound-obj}
	\| \tdH \B P^{(k)} \B y \|^2 \le\Big(1- \frac{1}{18R}\Big)^k \| \tdH \B P^{(0)} \B y \|^2 + 36 \|\tdH \B \ep \|^2  \ . 
	\end{equation}
	
\end{theorem}

In the special (noiseless) setting when $\B \ep = \B 0$, 
Theorem \ref{theorem: main result} establishes that the sequence of objective values generated by Algorithm~\ref{alg: AFW} converges to zero i.e., the optimal objective value, at a linear rate. The parameter for the linear rate of convergence depends upon the search width $R$. 
Following the discussion after Assumption~\ref{ass-main}, 
the sample-size requirement $n/\log(n) = \Omega(dr^2)$ is more stringent than that needed in order for the model to be identifiable ($n\ge 2d$)~\cite{unnikrishnan2018unlabeled} in the noiseless setting. In particular, when $ n/(d\log n) =  O(1)$, the number of mismatched pairs $r$ needs to be bounded by a constant.  
Numerical evidence presented in Section~\ref{sec:expts} (for the noiseless case) appears to suggest that the sample size $n$ needed to recover $\B P^*$ is smaller than what is suggested by our theory.

In the noisy case (i.e. $\B \ep\neq \B 0$), the bound~\eqref{upper-bound-obj} provides an upper bound on the objective value consisting of two terms. The first term converges to $0$ with a linear rate similar to the noiseless case. The second term is a constant multiple of the squared norm of the unavoidable noise term\footnote{Recall that the objective value at $\B P = \B P^*$ is $ \| \tdH \B P^* \B y \|^2= \| \tdH(\B X \B \beta^* + \B \ep) \|^2= \| \tdH \B \ep\|^2$.}: $ \| \tdH \B \ep \|^2 $. 
In other words, 
Algorithm \ref{alg: AFW} finds a solution whose objective value is at most a constant multiple of the objective value at the true permutation $\B P^*$.

	Theorem \ref{theorem: main result} proves a convergence guarantee on the objective value. The next result provides upper bounds on the $\ell_{\infty}$-norm of the mismatched entries i.e., $\| \B P^{(k)} \B y - \B P^* \B y  \|_{\infty}$. 
	For any $\B Q \in \Pi_n$, define 
	\begin{equation}
	G(\B Q) ~ := ~\| \tdH \B Q \B y \|^2 - \min_{\B P \in \cN_2(\B Q)\cap \cN_R(\B I_n)}  \| \tdH \B P \B y \|^2
	\end{equation}
	that is, $G(\B Q) $ is the decrease in the objective value after one step of local search starting at $\B Q$. For the permutation matrices 
	$\{\B P^{(k)}\}_{k \geq 0}$ generated by Algorithm \ref{alg: AFW}, 
	we know $G(\B P^{(k)}) = \| \tdH \B P^{(k)} \B y \|^2 - \| \tdH \B P^{(k+1)} \B y \|^2 $.

	\begin{theorem}\label{theorem: inf-norm-bound}
		($\ell_\infty$-bound on mismatched pairs)
		Suppose Assumptions~\ref{ass-main0} and \ref{ass-main} hold, and let $\{\B P^{(k)}\}_{k \geq 0}$ be the permutation matrices generated by 
		Algorithm~\ref{alg: AFW}. 
		Then for all $k\ge 0$ it holds 
		\begin{equation}\label{ineq: inf-norm-bound}
		\| \B P^{(k)} \B y - \B P^* \B y  \|_{\infty}^2 \le 800 \bar\sigma^2 + 10G(\B P^{(k)}) \ . \nonumber
		\end{equation}
	\end{theorem}
	Theorem \ref{theorem: inf-norm-bound} states that the largest squared error of the mismatched pairs (i.e., $\| \B P^{(k)} \B y - \B P^* \B y  \|_{\infty}^2$) is bounded above by a constant multiple of the one-step decrease in objective value (i.e. $G(\B P^{(k)})$) plus a term comparable to the noise level $O(\bar \sigma^2)$. In particular, if Algorithm \ref{alg: AFW} is terminated at an iteration $\B P^{(k)}$ with $ G(\B P^{(k)})$
	of the order of $\bar{\sigma}^2$, 
	then $\| \B P^{(k)} \B y - \B P^* \B y  \|^2_{\infty} $ is bounded by a constant multiple of $  \bar\sigma^2$.

	Note that the constant $800$ in \eqref{ineq: inf-norm-bound} is conservative and may be improved with a careful adjustment of the constants appearing in the proof and in the assumptions.

	In light of Theorem \ref{theorem: inf-norm-bound}, we can prove an upper bound on the estimation error of $\B \beta^*$, using an additional assumption stated below. 
	
	\begin{assumption}\label{ass-main2}
		There exists a constant $\bar\ga >0$ such that 
		\begin{equation}\label{ineq-ass3}
		(1) \quad \lam_{\min} \Big(\frac{1}{n} \B X^\T \B X\Big) \ge \bar\ga~~~\quad (2) \quad 
		\| (\B X^\T \B X)^{-1} \B X^\T \B\ep \| \le  \bar\sigma \sqrt{\frac{d}{n\bar\ga}} \ .
		\end{equation}
		where $\bar\sigma$ is as defined in Assumption \ref{ass-main0}. 
	\end{assumption}
	Section \ref{subsection: sufficient conditions for assumptions} presents conditions on $\B{X},\B{\epsilon}$ under which Assumption \ref{ass-main2} is satisfied with high probability.

	\begin{theorem}\label{theorem: consistency}
		(Estimation error)
		Suppose Assumptions \ref{ass-main0}, \ref{ass-main} and \ref{ass-main2} hold. Suppose iteration $k$ of Algorithm \ref{alg: AFW} satisfies $G(\B P^{(k)}) \le c\bar\sigma^2$ for a constant $c>0$. Let $\B \beta^{(k)} := (\B X^\T \B X)^{-1} \B X^\T  \B P^{(k)}  \B y$ and denote $\bar c:= 800+ 10c$. 
		Then we have 
		\begin{equation}\label{new-bound-1}
		\| \B\beta^{(k)} - \B\beta^* \|^2 ~\le~ 4\bar\ga^{-1}\bar c
		\frac{R  \bar\sigma^2 }{n}   + 2 \bar\ga^{-1}  \frac{d \bar\sigma^2}{n}
		\end{equation}
		and 
		\begin{equation}\label{new-bound-2}
		\frac{1}{n} \Big\| ( \B P^{(k)} )^{-1} \B X \B\beta^{(k)} - (\B P^*)^{-1} \B X \B\beta^* \Big\|^2 ~\le ~  2 (\sqrt{2\bar c}+3)^2 \frac{R\bar\sigma^2}{n} + \frac{2d\bar\sigma^2}{n} \ .
		\end{equation}
	\end{theorem}

	Theorem~\ref{theorem: consistency} (cf bound~\eqref{new-bound-1}) states that as long as $k$ is sufficiently large\footnote{We note that in Algorithm~\ref{alg: AFW}, as $k \rightarrow \infty$, the quantity $G(\B{P}^{(k)})\rightarrow 0$, and the condition $G(\B P^{(k)}) \le c\bar\sigma^2$ will hold for $k$ sufficiently large.}, the estimation error $\| \B\beta^{(k)} - \B\beta^* \|^2 $ is of the order $O(({r+d})\bar\sigma^2/n)$, assuming $\bar{\gamma}$ is a constant.  Therefore, as $n\rightarrow \infty$ (with $r,d$ fixed), the estimator delivered by our algorithm (after sufficiently many iterations) will converge to the true regression coefficient vector, $\B\beta^*$.
	In addition, \eqref{new-bound-2} provides an upper bound on the entrywise ``denoising error" (left hand side of \eqref{new-bound-2})---this is of the order $O((r+d)\bar\sigma^2/n)$. See \cite{pananjady2017denoising} for past works and discussions on this error metric.

	The following theorem provides an upper bound on the total number of local search steps needed to find a $\B P^{(k)}$ with $G(\B P^{(k)}) \le c \bar\sigma^2$. 
	
	\begin{theorem}\label{theorem: iteration-complexity}
		(Iteration complexity)
		Suppose Assumptions \ref{ass-main0} and \ref{ass-main} hold. Let $\{\B P^{(k)}\}_{k \geq 0}$ be the permutation matrices generated by 
		Algorithm~\ref{alg: AFW}. Given any $c>0$, define 
		\begin{equation}
		K^\dagger := \lt\lceil \log \Big( \frac{36\| \tdH \B\ep\|^2}{\| \tdH \B P^{(0)} \B y \|^2} \Big) \Big/ \log \Big( 1 - \frac{1}{18R} \Big)  ~+ ~ \frac{72n}{c}  \rt\rceil  +1  \ .
		\end{equation}
		Then there exists $0\le k \le K^\dagger$ such that $G(\B P^{(k)}) \le c\bar\sigma^2$. 
	\end{theorem}
	
	\begin{proof}
		Denote 
		\begin{equation}\label{K1}
		K_1 :=  \lt\lceil \log \Big( \frac{36\| \tdH \B\ep\|^2}{\| \tdH \B P^{(0)} \B y \|^2} \Big) \Big/ \log \Big( 1 - \frac{1}{18R} \Big)    \rt\rceil \ .
		\end{equation}
		Then by Theorem \ref{theorem: main result}, after $K_1$ iterations, it holds 
		\begin{equation}
		\| \tdH \B P^{(K_1)} \B y \|^2 \le     36 \| \tdH \B\ep\|^2 +  36\| \tdH \B\ep\|^2 =     72 \| \tdH \B\ep\|^2  \le 72 n \bar\sigma^2
		\end{equation}
		where the second inequality follows Assumption \ref{ass-main0} (4). 
		Suppose $G(\B P^{(k)}) > c \bar\sigma^2 $ for all $K_1 \le k \le K^\dagger-1$, then 
		\begin{equation}
		\| \tdH \B P^{(K^\dagger)} \B y \|^2 =   	\| \tdH \B P^{(K_1)} \B y \|^2 - \sum_{k=K_1}^{K^{\dagger}-1} G(\B P^{(k)}) <  72 n \bar\sigma^2   -  \frac{72n}{c}  c \bar\sigma^2 = 0 \ , \nonumber
		\end{equation}
		which is a contradiction. So there must exist some $K_1 \le k \le K^\dagger-1$ such that $G(\B P^{(k)}) \le c \bar\sigma^2 $. 
	\end{proof}
	
	Note that if $R$ and $\| \tdH \B\ep\|^2/ \| \tdH \B P^{(0)} \B y \|^2$ are bounded by a constant, then
	the number of iterations $K^\dagger= O(n)$. Therefore, in this situation, one can find an estimate $\B \beta^{(k)}$ satisfying $	\| \B\beta^{(k)} - \B\beta^* \|^2 \le O((d+r)\bar\sigma^2/n)$ within $O(n)$ iterations of Algorithm~\ref{alg: AFW}.

\subsection{Proofs of main theorems}\label{subsection: proofs of main theorems}
In this section, we present the proofs of Proposition \ref{prop: support-inclusion}, 
Theorem \ref{theorem: main result}, Theorem \ref{theorem: inf-norm-bound} and Theorem \ref{theorem: consistency}. 
We first present a technical result used in our proofs.
\begin{lemma}\label{lemma: large decrease in past iterations}
	Suppose Assumption \ref{ass-main0} holds. Let $\{\B P^{(k)}\}_{k \geq 0}$ be the permutation matrices generated by Algorithm \ref{alg: AFW}. 
	Suppose $\| \B P^{(k)} \B y - \B P^* \B y \|_{\infty} \ge L $  for some $ k\ge 1 $. Suppose at least one of the two conditions holds: (i) $k\le R/2$; or
	(ii)~$k\ge R/2 + 1$, and $\supp(\B P^*) \subseteq \supp(\B P^{(k')})$ for all $ R/2 \le k'\le k - 1$.
	Then for all $t\le k-1$, we have
	\begin{eqnarray}\label{ineq: constant decrease}
	\|  \B P^{(t+1)} \B  y - \B P^* \B y \|^2 - \|  \B P^{(t)} \B y - \B P^* \B y \|^2 \le -L^2/5 \ .
	\end{eqnarray}
\end{lemma}
The proof of Lemma~\ref{lemma: large decrease in past iterations} is presented in Section~\ref{section: proof lemma large decrease in past iterations}. 
As mentioned earlier, our analysis makes use of the one-step decrease condition in Lemma~\ref{lemma: one-step}. 
Note however, if the permutation matrix at the current iteration, denoted by $\B P^{(k)}$, is on the boundary, i.e. $\dist(\B P^{(k)}, \B I_n) = R$, it is not clear whether the permutation found by Lemma \ref{lemma: one-step} is within the search region $\cN_R(\B I_n)$. 
Lemma~\ref{lemma: large decrease in past iterations} helps address this issue (See the proof of Theorem \ref{theorem: main result} below for details).

\subsubsection{Proof of Proposition \ref{prop: support-inclusion}}

We show this result by contradiction. Suppose that there exists a $k\ge R/2$ such that $ \supp(\B P^{*}) \not\subseteq  \supp(\B P^{(k)})  $. Let $ T\ge R/2  $ be the first iteration ($\ge R/2$) such that $ \supp(\B P^{*}) \not\subseteq  \supp(\B P^{(T)})  $, i.e., 
$$
\supp(\B P^{*}) \not\subseteq  \supp(\B P^{(T)}) ~~~ {\rm and} ~~~
\supp(\B P^{*}) \subseteq  \supp(\B P^{(k)}) ~~ \forall~ R/2 \le k \le T-1 \ .
$$
Let $i\in \supp(\B P^{*})$ but $i\notin \supp(\B P^{(T)})$, then by Assumption \ref{ass-main0}~(1), we have 
$$
\| \B P^{(T)} \B y - \B P^* \B y \|_{\infty} \ge | \B e_i^\T (\B P^{(T)} \B y - \B P^* \B y) | = | \B e_i^\T ( \B y - \B P^* \B y)  | \ge L.
$$
By Lemma \ref{lemma: large decrease in past iterations}, we have $ \|  \B P^{(k+1)} \B y -  \B P^* \B y \|^2 - \|  \B P^{(k)} \B y -  \B P^* \B y \|^2 \le -L^2/5 $
for all $k\le T-1$. As a result, 
$$
\|  \B P^{(T)} \B y  -  \B P^* \B y\|^2 -  \|  \B P^{(0)} \B y  -  \B P^* \B y\|^2 \le -TL^2/5 \le -RL^2/10 \ .
$$
Since by Assumption~\ref{ass-main0}~(1),
$
\| \B P^{(0)} \B y - \B P^* \B y \|^2 =  \|   \B y - \B P^* \B y \|^2 \le r U^2 \ ,
$
we have
$$
\|  \B P^{(T)} \B y  -  \B P^* \B y\|^2 \le r U^2 - RL^2/10 \le r U^2 - \frac{L^2}{10} \frac{10C_1rU^2}{L^2} = (1-C_1) rU^2<0 \ .
$$
This is a contradiction, so such an iteration counter $T$ does not exist; and for all $k\ge R/2$, we have $ \supp(\B P^{*}) \subseteq  \supp(\B P^{(k)})  $.

\subsubsection{Proof of Theorem \ref{theorem: main result}}
Let $a = 2.2$. 
Because $R\ge 10r$, we have $r +R \le 1.1R = aR/2$; 
and for any $k\ge 0$: 
$$\| \B P^{(k)} \B y - \B P^* \B y \|_0 \le \dist (  \B P^{(k)} , \B I_n ) + \dist (  \B P^* , \B I_n ) \le R+r\le aR/2.$$ 
Hence, by Lemma~\ref{lemma: one-step}, there exists a permutation matrix $\tdP^{(k)} \in \Pi_n$ such that $\dist ( \tdP^{(k)} , \B P^{(k)}  ) \le 2  $, $\supp ( \tdP^{(k)} (\B P^*)^{-1}) \subseteq   \supp ( \B  P^{(k)} (\B P^*)^{-1} ) $ and 
$$
\| \tdP^{(k)} \B y- \B P^* \B y\|^2 \le (1-1/(aR) ) \| \B P^{(k)} \B y - \B P^* \B y \|^2  \ . 
$$
As a result, 
\begin{eqnarray}
\| \tdH (\tdP^{(k)} \B  y- \B P^* \B y) \|^2 \le 
\|  \tdP^{(k)} \B y- \B P^* \B y \|^2 &\le& 
(1-1/(aR) ) \| \B P^{(k)} \B y - \B P^* \B y \|^2  \nonumber\\
&\le& 
\frac{1-1/(aR)}{1-R\rho_n} \| \tdH(\B P^{(k)} \B y - \B P^* \B y) \|^2 \ , \nonumber
\end{eqnarray}
where the last inequality is from Assumption~\ref{ass-main0}~(3).
Note that by Assumption~\ref{ass-main}~(1), we have $R^2 \rho_n \le 1/10 \le 1/(2a) $, so
$  R\rho_n \le  1/(2aR) $. Because $ (1- 1/(aR)) \le (1-1/(2aR))^2 $, we have 
\begin{eqnarray}
\| \tdH (\tdP^{(k)} \B y- \B P^* \B y) \|^2 
&\le& \frac{1-1/(aR)}{1-1/(2aR)} \| \tdH(\B P^{(k)} \B y - \B P^*\B y) \|^2 \nonumber\\
&\le&  (1-1/(2aR)) \| \tdH(\B P^{(k)} \B y - \B P^*\B y) \|^2 \ . \nonumber
\end{eqnarray}
Recall that~\eqref{model1} leads to $\tdH \B P^* \B y = \tdH \B \ep$, so we have
\begin{equation}\label{ineq-10}
\| \tdH \tdP^{(k)} \B y- \tdH \B \ep \|^2 \le  (1-1/(2aR)) \| \tdH \B P^{(k)} \B y - \tdH \B \ep \|^2 \ .
\end{equation}
Let $\eta := 1/(2aR)$ and $\B z:= \tdP^{(k)} \B y - \B P^{(k)} \B y$, then \eqref{ineq-10} leads to:
\begin{eqnarray}\label{ineq-c1}
&&\| \tdH \tdP^{(k)} \B y \|^2 \nonumber\\
&\le& 	(1-\eta )\| \tdH  \B P^{(k)} \B y \|^2 - \eta \| \tdH \B \ep \|^2 + 2 \la \tdH \tdP^{(k)} \B y, \tdH \B \ep \ra - 2(1-\eta) \la \tdH \B P^{(k)} \B y, \tdH \B \ep \ra \nonumber\\
&\le&
(1-\eta )\| \tdH  \B P^{(k)} \B y \|^2 + 2 (1-\eta) \la \tdH \B z , \tdH \B \ep \ra + 2\eta \la \tdH \tdP^{(k)} \B y, \tdH \B \ep \ra \nonumber\\
&\le&
(1-\eta )\| \tdH  \B P^{(k)} \B y \|^2 +
2|\la \tdH \B z , \tdH \B \ep \ra | +
2\eta |\la \tdH \tdP^{(k)} \B y, \tdH \B \ep \ra| \ 
\end{eqnarray}
where, to arrive at the second inequality, we drop the term $- \eta \| \tdH \B \ep \|^2$.
We now make use of the following claim whose proof is in 
Section~\ref{subsection: proof of claim-bound}:
\begin{equation}\label{claim-bound}
{\bf Claim.} ~~~~ 2|\la \tdH \B z , \tdH \B \ep \ra | \le 
\frac{\eta}{4}\|\tdH \tdP^{(k)} \B y\|^2 + \frac{\eta}{4}\|\tdH  \B P^{(k)} \B y\|^2 + 4\eta \|\tdH \B \ep \|^2 \ .
\end{equation}
On the other hand, by Cauchy-Schwarz inequality, 
\begin{equation}\label{ineq-c3}
2\eta |\la \tdH \tdP^{(k)} \B y, \tdH \B \ep \ra| \le 
({\eta}/{4}) \| \tdH \tdP^{(k)} \B y \|^2 + 4\eta  \| \tdH \B \ep\|^2 \ .
\end{equation}
Combining \eqref{ineq-c1}, \eqref{claim-bound} and \eqref{ineq-c3},
we have
\begin{equation}
\begin{aligned}
\| \tdH \tdP^{(k)} \B y \|^2  &\le 
(1-\eta )\| \tdH  \B P^{(k)} \B y \|^2 + (\eta/2) 	\| \tdH \tdP^{(k)} \B y \|^2 \\
&~~~ + (\eta/4) 	\| \tdH \B P^{(k)} \B y \|^2  
+8 \eta \|\tdH \B \ep \|^2 \ . 
\end{aligned}
\nonumber
\end{equation}
After some rearrangement, the above leads to:
\begin{equation}\label{eq-above-claim}
\begin{aligned}
\| \tdH \tdP^{(k)} \B  y \|^2  &\le \frac{1-3\eta/4}{1-\eta/2}  \| \tdH  \B P^{(k)} \B y \|^2+\frac {8 \eta}{1-\eta/2}  \|\tdH \B \ep \|^2 \\
& \le 
(1-\eta/4) \| \tdH  \B P^{(k)} \B y \|^2+9 \eta \|\tdH \B \ep \|^2 \ 
\end{aligned}
\end{equation}
where the second inequality uses $ 1-3\eta/4 \le (1-\eta/2) (1-\eta/4)$ and $ (1-\eta/2)^{-1} \le 9/8 $ (recall, $\eta = 1/(2aR)$).

\noindent To complete the proof, we use another claim whose proof is in Section~\ref{subsection: proof of claims}:
\begin{equation}\label{claim2}
{\bf Claim.} ~~~ {\rm For~ any~ } k\ge 0 {\rm ~ it ~ holds~ that~} \tdP^{(k)} \in \cN_R(\B I_n) \cap \cN_{2}(\B P^{(k)})  \ .
\end{equation}
By the above claim, the update rule \eqref{update} and inequality \eqref{eq-above-claim}, we have
$$
\| \tdH \B P^{(k+1)} \B y\|^2 \le 	\| \tdH \tdP^{(k)} \B y \|^2  \le 	(1-\eta/4) \| \tdH  \B P^{(k)} \B y \|^2+9 \eta \|\tdH \B \ep \|^2 \ .
$$
Using the notation $a_k:= \| \tdH \B P^{(k)} \B y \|^2$, $\lam =1-\eta/4$ and $ \tilde{e} = 9 \eta \|\tdH \B \ep \|^2 $, the above inequality leads to:
$ a_{k+1} \le \lam a_k + \tilde{e}$ for all $k\ge 0$. Therefore, we have
\begin{equation}
\frac{a_{k+1}}{\lam^{k+1}} \le \frac{a_{k}}{\lambda
	^{k}} + \frac{\tilde{e}}{\lam^{k+1}} \le \frac{a_{k-1}}{\lam^{k-1}} + \frac{\tilde{e}}{\lam^{k}}  + \frac{\tilde{e}}{\lam^{k+1}} \le \cdots \le \frac{a_0}{\lam^0} + \tilde{e} \sum_{i=1}^{k+1} \frac{1}{\lam^i} \ , \nonumber
\end{equation}
which implies $ a_k \le a_0 \lam^k + \tilde{e} \sum_{i=1}^k \lam^{i-1} \le 
a_0 \lam^k + ({\tilde{e}}/{(1-\lam)}) $.
This leads to 
\begin{equation}
\begin{aligned}
\| \tdH \B P^{(k)} \B y \|^2 &\le 	
(1-\eta/4)^k
\| \tdH \B P^{(0)} \B y \|^2  + \frac{9 \eta \|\tdH \B \ep \|^2}{\eta/4} \\
&\le
(1- {1}/{(8aR)})^k \| \tdH \B P^{(0)}\B  y \|^2 + 36 \|\tdH \B \ep \|^2 \ . 
\end{aligned}
\nonumber
\end{equation}
Recalling that $ 8a \le 18 $, we conclude the proof of the theorem.

	\subsubsection{Proof of Theorem \ref{theorem: inf-norm-bound}}
	By the definition of $G(\cdot)$, we have  
	\begin{equation}\label{local minimum}
	\| \tdH \B P^{(k)} \B y \|^2 \le 	\| \tdH \B P \B y \|^2 + G(\B P^{(k)}) ~~~ \forall ~ \B P\in \cN_2(\B P^{(k)}) \cap \cN_R(\B I_n) \ . 
	\end{equation}
	By Lemma \ref{lemma: one-step}, there exists a permutation matrix $\tdP^{(k)} \in \Pi_n$ such that 
	$$\dist ( \tdP^{(k)} ,\B P^{(k)}  ) \le 2, ~~\supp ( \tdP^{(k)} (\B P^*)^{-1}) \subseteq   \supp (  \B P^{(k)} (\B P^*)^{-1} )$$ 
	and 
	\begin{equation}\label{decrease}
	\| \tdP^{(k)} \B y- \B P^* \B y\|^2 - \| \B P^{(k)} \B y - \B P^*\B y \|^2  \le - (1/2) \| \B P^{(k)} \B y - \B P^*\B y \|^2_{\infty} \ . 
	\end{equation}
	By Claim \eqref{claim2} we have
	\begin{equation}\label{claim1}
	\tdP^{(k)} \in \cN_R(\B I_n) \cap \cN_{2}(\B P^{(k)}) \ .
	\end{equation}
	Therefore, by \eqref{local minimum} and \eqref{claim1}, we have 
	$$
	\| \tdH \B P^{(k)} \B  y \|^2 \le 	\| \tdH \tdP^{(k)} \B y \|^2 + G(\B P^{(k)})\ .
	$$
	Let $\B z := \tdP^{(k)} \B y - \B P^{(k)} \B y$. 
	Recall that $\tdH \B P^* \B y = \tdH \B \ep$, so by the inequality above we have
	$$
	\| \tdH (\B P^{(k)} \B y - \B P^* \B y) + \tdH \B \ep  \|^2 \le 	\| \tdH (\B P^{(k)} \B y - \B P^* \B y) + \tdH \B \ep + \tdH \B z  \|^2 +G(\B P^{(k)}) 
	$$
	which is equivalent to 
	\begin{equation}\label{sum1}
	-2 \la \tdH (\B P^{(k)} \B y - \B P^*\B y) , \tdH \B z \ra - \| \tdH \B z \|^2 \le 2 \la \tdH \B z, \tdH \B \ep \ra  +  G(\B P^{(k)})  \ . 
	\end{equation}
	On the other hand, from \eqref{decrease} we have 
	$$
	\| \B P^{(k)} \B y - \B P^* \B y + \B z \|^2  - \| \B P^{(k)} \B y - \B P^* \B y  \|^2\le - (1/2) \| \B P^{(k)} \B y - \B P^* \B y  \|_{\infty}^2  \ 
	$$
	or equivalently,
	\begin{equation}\label{sum2}
	2 \la  \B P^{(k)} \B y - \B P^* \B y,  \B z \ra + \| \B z\|^2 \le - (1/2) \| \B P^{(k)} \B y - \B P^* \B y  \|_{\infty}^2 \ .
	\end{equation}
	Summing up \eqref{sum1} and \eqref{sum2} we have 
	\begin{equation}\label{ineq: basic}
	\begin{aligned}
	&~ 2 \la  \B H( \B P^{(k)} \B y - \B P^* \B y), \B H \B z \ra + 
	\| \B H  \B z\|^2 \\
	\le & ~ 2 \la \tdH  \B z, \tdH \B \ep \ra  - (1/2) \| \B P^{(k)} \B y - \B P^* \B y  \|_{\infty}^2 +  G(\B P^{(k)}) \ . 
	\end{aligned}
	\end{equation}
	Note that 
	\begin{equation}
	\begin{aligned}
	& \ 2 | \la  \B H( \B P^{(k)} \B y - \B P^* \B y), \B H \B z \ra  |  \le 
	2 \| \B H( \B P^{(k)} \B y - \B P^* \B y) \| \cdot \| \B H \B z  \|  
	\\
	\le & \ 2 \sqrt{R\rho_n} \| \B P^{(k)} \B y - \B P^* \B y \| 
	\| \B H \B z  \|  \le 
	2\sqrt{2} R \sqrt{\rho_n} \| \B P^{(k)} \B y - \B P^* \B y \|_{\infty} 
	\| \B H \B z  \| 
	\end{aligned}
	\end{equation}
	where the second inequality is 
	by Assumption \ref{ass-main0} (3) and the third inequality uses
	$\| \B P^{(k)} \B y - \B P^* \B y  \|_0 \le 2R$. 
	From Assumption \ref{ass-main} (1) we have $R\sqrt{\rho_n} \le 1/\sqrt{10}$, hence 
	\begin{equation}\label{ineq-new1}
	\begin{aligned}
	& ~
	2 | \la  \B H( \B P^{(k)} \B y - \B P^* \B y), \B H \B z \ra  | \le \frac{2}{\sqrt{5}} \| \B P^{(k)} \B y - \B P^* \B y \|_{\infty} 
	\| \B H \B z  \| 
	\\
	\le &~ \frac{1}{5} \| \B P^{(k)} \B y - \B P^* \B y \|_{\infty}^2 + 
	\| \B H \B z  \|^2  
	\end{aligned}
	\end{equation}
	where the last inequality is by Cauchy-Schwarz inequality. 
	Rearranging terms in~\eqref{ineq: basic}, and making use of~\eqref{ineq-new1}, we have 
	\begin{eqnarray}
	&& \| \B H  \B z\|^2  +  (1/2) \| \B P^{(k)} \B y - \B P^* \B y  \|_{\infty}^2 
	\nonumber\\
	&\le&  2 \la \tdH  \B z, \tdH \B \ep \ra -
	2 \la  \B H( \B P^{(k)} \B y - \B P^* \B y), \B H \B z \ra +  G(\B P^{(k)}) \nonumber\\
	&\le&
	2 \la \tdH  \B z, \tdH \B \ep \ra 
	+ (1/5) \| \B P^{(k)} \B y - \B P^* \B y \|_{\infty}^2 + 
	\| \B H \B z  \|^2 +  G(\B P^{(k)}) \ .  \nonumber
	\end{eqnarray}
	As a result, 
	\begin{equation}\label{ineq-new2}
	\begin{aligned}
	(3/10) \| \B P^{(k)} \B y - \B P^* \B y  \|_{\infty}^2  \le& 
	2 \la \tdH  \B z, \tdH \B \ep \ra +  G(\B P^{(k)}) \\
	=& 
	2\la   \B z, \tdH \B \ep \ra  +  G(\B P^{(k)}).
	\end{aligned}
	\end{equation}
	By the definition of $\B z$, we know there exist $i,j \in [n]$ such that 
	$$\B z = \|\B z\|_{\infty} (\B e_i - \B e_j) = (\| \B z \|/\sqrt{2})(\B e_i - \B e_j).$$ Therefore 
	\begin{equation}\label{ineq-new3}
	2 \la   \B z, \tdH \B \ep \ra = \sqrt{2}  \|\B z\| \la \B e_i - \B e_j, \tdH \B \ep  \ra \le 2\sqrt{2} \|\B z\|  \| \tdH \B \ep\|_{\infty} \le 
	2\sqrt{2}  \bar \sigma \|\B z\| 
	\end{equation}
	where the last inequality makes use of Assumption \ref{ass-main0} (4). 
	On the other hand, by \eqref{sum2} we have 
	\begin{eqnarray}
	\|\B z\|^2  \le 
	2 |\la  \B P^{(k)} \B y - \B P^* \B y,  \B z \ra|
	=
	\sqrt{2} \| \B z \|   |\la  \B P^{(k)} \B y - \B P^* \B y,  \B e_i - \B e_j \ra|  \ ,  \nonumber
	\end{eqnarray}
	and hence 
	\begin{equation}\label{ineq-new4}
	\|\B z\| \le \sqrt{2}   |\la  \B P^{(k)} \B y - \B P^* \B y,  \B e_i - \B e_j \ra |
	\le 2\sqrt{2}  \| \B P^{(k)} \B y - \B P^* \B y  \|_{\infty} \ .
	\end{equation}
	Combining \eqref{ineq-new2}, \eqref{ineq-new3} and \eqref{ineq-new4}, we have 
	\begin{eqnarray}
	(3/10) \| \B P^{(k)} \B y - \B P^* \B y  \|_{\infty}^2  &\le&  
	8 \bar\sigma \| \B P^{(k)} \B y - \B P^* \B y  \|_{\infty}  +  G(\B P^{(k)})
	\label{proof-thm4-last-1}\\
	&\le&
	80 \bar\sigma^2 + (1/5) \| \B P^{(k)} \B y - \B P^* \B y  \|_{\infty}^2  + G(\B P^{(k)})  \nonumber
	\end{eqnarray}
	where the second inequality is by Cauchy-Schwarz inequality. Inequality in display~\eqref{proof-thm4-last-1} leads to
	\begin{equation}
	\| \B P^{(k)} \B y - \B P^* \B y  \|_{\infty}^2 \le 800 \bar\sigma^2 + 10G(\B P^{(k)}) \ \nonumber
	\end{equation}
	which completes the proof of this theorem.

	\subsubsection{Proof of Theorem \ref{theorem: consistency}}
	
	Recall that $\B P^* \B y = \B{X} \B\beta^* + \B\ep$, so it holds 
	$\B \beta^* = (\B X^\T \B X)^{-1} \B X^\T (\B P^* \B y - \B \ep )$. Therefore
	\begin{equation}\label{beta-diff-1}
	\B\beta^{(k)} - \B\beta^* = (\B X^\T \B X)^{-1} \B X^\T (  \B P^{(k)} \B y - \B P^* \B y + \B \ep) \ .
	\end{equation}
	Hence we have 
	\begin{equation}\label{a1}
	\begin{aligned}
	\| \B\beta^{(k)} - \B\beta^* \| &\le
	\| (\B X^\T \B X)^{-1} \B X^\T (  \B P^{(k)} \B y - \B P^* \B y ) \| + 
	\|  (\B X^\T \B X)^{-1} \B X^\T \B \ep  \|     \\
	&\le \| (\B X^\T \B X)^{-1} \B X^\T (  \B P^{(k)} \B y - \B P^* \B y ) \| + 
	\bar \sigma \sqrt{{d}/{(n\bar\ga)}} 
	\end{aligned}
	\end{equation}
	where the second inequality is by Assumption \ref{ass-main2} (2). 
	Note that
	\begin{equation}\label{a2}
	\| \B P^{(k)} \B y - \B P^* \B y  \| \le \sqrt{2R} \| \B P^{(k)} \B y - \B P^* \B y  \|_{\infty} \le 
	\bar\sigma
	\sqrt{2R\bar c} 
	\end{equation}
	where the first inequality is because $\B P^{(k)} \B y - \B P^* \B y$ has at most $2R$ non-zero coordinates; the second inequality makes use of Theorem \ref{theorem: inf-norm-bound} and the definition of $\bar c$ in Theorem \ref{theorem: consistency}. 
	On the other hand, 
	by Assumption \ref{ass-main2} (1) we have 
	\begin{equation}\label{a3}
	\itl (\B X^\T \B X)^{-1} \B X^\T    \itl_2 = \sqrt{\itl (\B X^\T \B X)^{-1} \itl_2} \le 1/\sqrt{n\bar\ga} \ .
	\end{equation}
	Combining \eqref{a1}, \eqref{a2} and \eqref{a3} we have 
	\begin{eqnarray}
	\| \B\beta^{(k)} - \B\beta^* \| 
	&\le& \itl (\B X^\T \B X)^{-1} \B X^\T    \itl_2   \| \B P^{(k)} \B y - \B P^* \B y  \| + 
	\bar\sigma \sqrt{{d}/{(n\bar\ga)}}  \nonumber\\
	&\le& 
	\bar\sigma \sqrt{2\bar\ga^{-1}R\bar c/n}  +    \bar \sigma \sqrt{{d\bar\ga^{-1}}/{n}}  \ . \nonumber
	\end{eqnarray}
	Squaring both sides of the above, we get
	\begin{equation}
	\| \B\beta^{(k)} - \B\beta^* \|^2 ~\le~
	4\bar\ga^{-1}\bar c
	\frac{R  \bar\sigma^2 }{n}   + 2 \bar\ga^{-1}  \frac{d \bar\sigma^2}{n} \  \nonumber
	\end{equation}
	which completes the proof of \eqref{new-bound-1}. 
	
	We will now prove \eqref{new-bound-2}. Let us denote $\B J:= ( \B P^{(k)} )^{-1} \B X \B\beta^{(k)} - (\B P^*)^{-1} \B X \B\beta^*$. 
	Note that we can write
	\begin{equation}\label{b1}
	\B J ~=~  	( \B P^{(k)} )^{-1} \B X (\B\beta^{(k)} - \B\beta^*) + ( ( \B P^{(k)} )^{-1} - (\B P^*)^{-1} ) \B X \B\beta^* \ .
	\end{equation}
	Multiplying both sides of~\eqref{beta-diff-1} by $(\B P^{(k)} )^{-1} \B X$, we have 
	\begin{equation}\label{b2}
	( \B P^{(k)} )^{-1} \B X (\B\beta^{(k)} - \B\beta^*) = ( \B P^{(k)} )^{-1} \B H (  \B P^{(k)} \B y - \B P^* \B y + \B \ep) \ .
	\end{equation}
	On the other hand, 
	\begin{equation}\label{b3}
	\begin{aligned}
	&~
	( ( \B P^{(k)} )^{-1} - (\B P^*)^{-1} ) \B X \B\beta^* \\
	=&~
	( ( \B P^{(k)} )^{-1} - (\B P^*)^{-1} ) (\B P^* \B y - \B \ep) \\
	=&~ (\B P^{(k)})^{-1} ( \B P^* \B y - \B P^{(k)} \B y) - 
	( ( \B P^{(k)} )^{-1} - (\B P^*)^{-1} ) \B \ep \ .
	\end{aligned}
	\end{equation}
	Combining \eqref{b1}, \eqref{b2} and \eqref{b3} we have 
	\begin{equation}\label{J}
	\begin{aligned}
	\B J &=  \underbrace{( \B P^{(k)} )^{-1} \tdH (   \B P^* \B y -  \B P^{(k)} \B y )}_{:=\B J_{1}}~+~
	\underbrace{( \B P^{(k)} )^{-1} \B H  \B \ep}_{:=\B{J}_{2}} ~+~ 
	\underbrace{(  (\B P^*)^{-1} - ( \B P^{(k)} )^{-1}  ) \B \ep}_{:=\B{J}_{3}} \\
	&= \B J_1 +  \B J_2 +   \B J_3 \ .
	\end{aligned}
	\end{equation}
	By \eqref{a2}, we know 
	\begin{equation}\label{J1}
	\| \B J_1 \| \le  \|    \B P^* \B y -  \B P^{(k)} \B y  \| \le  \bar\sigma
	\sqrt{2R\bar c} \ .
	\end{equation}
	By Assumption \ref{ass-main0} (4) we have 
	\begin{equation}\label{J2}
	\| \B J_2 \| = \| \B H  \B \ep \| \le \sqrt{d} \bar\sigma \ .
	\end{equation}
	Since $\dist((\B P^*)^{-1} , ( \B P^{(k)} )^{-1}) \le 2R$, it holds 
	\begin{equation}\label{J3}
	\| \B J_3 \| = \|   (\B P^*)^{-1} \B\ep - ( \B P^{(k)} )^{-1} \B\ep  \| \le 2\sqrt{2R} \| \B \ep\|_\infty \le 3 \sqrt{R} \bar\sigma
	\end{equation}
	where the last inequality makes use of Assumption \ref{ass-main0} (4). 
	
	Using \eqref{J1}, \eqref{J2} and \eqref{J3}, to bound the r.h.s of~\eqref{J}, we have 
	\begin{equation}
	\| \B J \| ~ \le ~ 
	\bar\sigma
	\sqrt{2R\bar c} +  \sqrt{d} \bar\sigma  +  3 \sqrt{R} \bar\sigma \ .  \nonumber
	\end{equation}
	As a result, 
	\begin{equation}
	\frac{1}{n} \| \B J \|^2 ~ \le ~  2 (\sqrt{2\bar c}+3)^2 \frac{R\bar\sigma^2}{n} + \frac{2d\bar\sigma^2}{n}  \  . \nonumber
	\end{equation}
	This completes the proof of \eqref{new-bound-2}.

\subsection{Sufficient conditions for assumptions to hold}\label{subsection: sufficient conditions for assumptions}
Our analysis in Sections~\ref{subsection: Main results} and~\ref{subsection: proofs of main theorems} was completely deterministic in nature under Assumptions~\ref{ass-main0},~\ref{ass-main} and \ref{ass-main2}.  
To provide some intuition, 
in the following, we discuss some probability models on $\B X$ and $\B \ep$ under which Assumption~\ref{ass-main0}~(3),~(4),  Assumption~\ref{ass-main}~(2) and Assumption \ref{ass-main2}
hold true with high probability.

\subsubsection{A random model matrix $\B X$}
When the rows of $\B X$ are iid draws from a well behaved probability distribution, Assumption \ref{ass-main0} (3) and
Assumption \ref{ass-main2} (1) 
hold true with high probability. This is formalized via the following lemma.
\begin{lemma}\label{lemma: rip}
	Suppose the rows of the matrix $\B X$: $\B x_1,\ldots,\B x_n$ are iid zero-mean random vectors in $\R^d$ with covariance matrix $\B \Sigma\in \R^{d\times d}$. Suppose there exist constants $\gamma, b,V>0$ such that $ \lam_{\min} (\B \Sigma) \ge \gamma $, $ \| \B x_i \| \le b $ and $ \| \B x_i \|_{\infty} \le V $ almost surely. 
	Given any $\tau>0$, 
	define
	$$
	\dt_{n,m}:= 16V^2  \Big(\frac{d}{n\ga}\log(2d/\tau)   + \frac{dm}{n\ga} \log(3n^2) \Big) \ .
	$$
	Suppose $n$ is large enough such that $\sqrt{\dt_{n,m}} \ge 2/n$ and $ 3b^2 \itl \B \Sigma\itl_2 \log(2d/\tau) / n\le (1/4)\ga^2  $. 
	Then with probability at least $ 1-2\tau$, it holds $\lam_{\min} (\frac{1}{n} \B X^\T \B X ) \ge \ga/2$, and 
	\begin{equation}
	\| \B H \B u \|^2  \le  {\dt_{n,m}} \| \B u\|^2  ~~~~ \forall ~\B u\in \cB_m \ . 
	\end{equation}
\end{lemma}

The proof of Lemma \ref{lemma: rip} is presented in Section \ref{subsection: proof of rip lemma}. 
Suppose there are universal constants $\bar c>0$ and $\bar C>0$ such that the parameters $(\gamma, V, b, \itl \B \Sigma \itl_2 ,\tau )$ in Lemma~\ref{lemma: rip} satisfy $\bar c \le \gamma, V, b, \itl \B \Sigma \itl_2 ,\tau \le \bar C$. 
Given a pre-specified probability level (e.g., $1-2\tau=0.99$), 
under the setting of Lemma \ref{lemma: rip}, 
if we set $\rho_n = \dt_{n,4}$ and $\bar\sigma = \ga/2$, then Assumption~\ref{ass-main0}~(3) and Assumption \ref{ass-main2} (1) are true with high probability ($\ge 1-2\tau$). 

Note that the almost sure boundedness assumption on $\|\B x_i\|$ 
can be relaxed to cases when $\|\B x_i\|$ is bounded with high probability (e.g. $\B x_i \stackrel{\text{iid}}{\sim} N(\B 0, \B\Sigma)$).

\subsubsection{The error distribution}
In the following, we discuss a commonly used random setting under which Assumption \ref{ass-main0} (4) and Assumption \ref{ass-main} (2) hold with high probability. 
A random variable $\xi$ is called sub-Gaussian \cite{wainwright2019high} with variance proxy $\vartheta^2$ (denoted by $\xi \in \subG(\vartheta^2)$) if $\Ex[\xi] = 0$ and $\Ex e^{t\xi} \le e^{\vartheta^2 t^2 /2}$ for all $t\in \R$.

	\begin{lemma}\label{lemma: noise}
		Suppose
		$\B \ep = [\ep_1,...,\ep_n]^\T$ with
		$\ep_1,...,\ep_n 
		\stackrel{\text{iid}}{\sim} \subG(\sigma^2)$ for some $\sigma>0 $. Suppose $\B \ep$ is independent of $\B X$. Then with probability\footnote{The probability statements here are conditional on $\B X$.} at least $ 1-\tau $ it holds
		\smallskip
		
		(a) $\max\{\| \B \ep \|_{\infty},  \| \tdH \B \ep \|_{\infty}  \} \le \sigma \sqrt{2\log(6n/\tau)}$. 
		\smallskip
		
		(b) $	\| \B H \B \ep \|\le  \sigma \sqrt{2d\log(6d/\tau)} $. 
		\smallskip
		
		(c) 	$	\| (\B X^\T \B X)^{-1} \B X^\T \B\ep \| \le    \lam^{-1/2}_{\min} ( \frac{1}{n} \B X^\T \B X ) \cdot  \sigma \sqrt{2d \log(6d/\tau)/n}$. 
		\smallskip
		
		In addition, if $\wtd \sigma^2 := \text{Var}(\ep_i) \ge (3/4) \sigma^2$, then there exists a universal constant $C>0$ such that if $\sqrt{\log(4/\tau)/(Cn)} + 2d \log(4d/\tau)/n \le 1/4 $, then with probability at least $1-\tau$, 
		\begin{equation}\label{conclusion2-noise}
		\| \tdH \B \ep \|^2 \ge (1/2)n\sigma^2 \ . 
		\end{equation}
	\end{lemma}

The proof of Lemma~\ref{lemma: noise} is presented in Section \ref{section: proof of lemma noise}. Note that in Lemma~\ref{lemma: noise} the assumption $\text{Var}(\ep_i) \ge (3/4) \sigma^2$ can be replaced by $ \text{Var}(\ep_i) \ge C_0 \sigma^2 $ for any constant $C_0$, with the conclusion changing accordingly (i.e., $1/2$ in \eqref{conclusion2-noise} will be replaced by another constant). In particular, Lemma \ref{lemma: noise} holds true when $\ep_i \stackrel{\text{iid}}{\sim} N(0,\sigma^2)$. 
Given a pre-specified probability level (e.g., $1-\tau=0.99$), 
under the setting of Lemma \ref{lemma: noise}, if we set 
$\bar\sigma = \sigma \sqrt{2\log(6n/\tau)}$, then Assumption~\ref{ass-main0}~(4), Assumption~\ref{ass-main}~(2) and Assumption \ref{ass-main2}~(2) hold with probability at least $1-2\tau$.

	\subsubsection{Summary}
	
	We summarize parameter choices informed by the results above in the following corollary.

	\begin{corollary}
		Suppose the matrix $\B X$ is drawn from a probability model as discussed in
		Lemma~\ref{lemma: rip} with $m=R$; and
		the noise term $\B \ep$ satisfies the assumptions in Lemma~\ref{lemma: noise}. Then with probability at least $1-6\tau$,
		the inequalities in \eqref{ass-re}, \eqref{ineq-ass1}, \eqref{ineq-ass2} and \eqref{ineq-ass3} hold true with the following parameters 
		\begin{equation}
		\rho_n = 16V^2  \Big(\frac{d}{n\ga}\log(2d/\tau)   + \frac{4d}{n\ga} \log(3n^2) \Big), 
		\end{equation}
		$\bar\sigma = \sigma \sqrt{2\log(6n/\tau)}$ and $\bar\ga = \ga/2$. 
		
		If in addition the conditions in Assumption \ref{ass-main0} (1) and (2) hold true and the following four inequalities 
		\begin{equation}
		R \rho_n \le L^2/(90U^2), \quad \bar \sigma \le \min\{0.5, (\rho_nd)^{-1/2}\} L^2/(80U) \ ,  \nonumber
		\end{equation}
		\begin{equation}
		R^2 \rho_n \le 1/10, \quad \bar \sigma^2  \le \sigma^2 \min \{n/(660R^2), \  n/(5dR)\} \ 
		\end{equation}
		are true,
		then all the statements in Assumptions \ref{ass-main0}, \ref{ass-main} and \ref{ass-main2} are satisfied. 
	\end{corollary}

\section{Approximate local search steps for computational scalability}\label{section: fast local search steps}
As discussed in Section~\ref{sec:local-search-sec},  the local search step~\eqref{update} in Algorithm~\ref{alg: AFW} costs $O(n^2)$ for each iteration $k$---this can limit the scalability of Algorithm~\ref{alg: AFW} to problems with a large $n$. Here we discuss an efficient method to find an approximate solution for step~\eqref{update}. Suppose that in the $k$-th iteration of Algorithm \ref{alg: AFW} the permutation $\B P^{(k)}$ satisfies
$ \dist(\B P^{(k)} , \B I_n) \le R-2 $, 
then update \eqref{update} is
\begin{equation}\label{prob1}
\min_{\B P}~~~ \| \tdH \B P \B y \|^2 ~~~ \text{s.t. }~~~ \dist(\B P, \B P^{(k)} ) \le 2. 
\end{equation}
Problem~\eqref{prob1} can be equivalently formulated as 
\begin{align}
& \min_{i,j\in [n]}  \| \tdH (\B P^{(k)} \B y - (y_i - y_j) (\B e_i- \B e_j) ) \|^2 \label{prob2} \\
=&  \min_{i,j\in[n]} \Big\{ \| \tdH \B P^{(k)} \B y \|^2 - 2 (y_i - y_j) \la \B e_i - \B e_j , \tdH \B P^{(k)} \B y \ra + (y_i - y_j)^2 \| \tdH (\B e_i - \B e_j) \|^2 \Big\}. \nonumber 
\end{align}
In view of Assumption \ref{ass-main0} (3), for $n\gg d$, we have $ \| \tdH(\B e_i - \B e_j) \|^2 \approx 2 $. 
Note that in general, $\| \tdH(\B e_i - \B e_j) \|^2 \leq 2$. Hence, one can approximately optimize~\eqref{prob2} by minimizing an upper bound of the last two terms in the second line of display~\eqref{prob2}. This is given by:
\begin{equation}\label{prob3}
\min_{i,j\in[n]} ~~  \Big\{- 2 (y_i - y_j) \la \B e_i - \B e_j , \tdH \B P^{(k)} \B y \ra + 2 (y_i - y_j)^2 \Big\}. 
\end{equation}
Denoting $\B w := \tdH \B P^{(k)} \B y$ and $\B v: = \B y - \B w$, the objective in~\eqref{prob3} is given by
\begin{eqnarray}
&&
- 2 (y_i - y_j) \la \B e_i - \B e_j , \tdH \B P^{(k)} \B y \ra + 2 (y_i - y_j)^2 \nonumber\\
&=&
2(y_i - y_j) (-w_i + w_j  + y_i - y_j) = 2(y_i - y_j) (v_i - v_j) \ . \nonumber
\end{eqnarray}
So problem \eqref{prob3} is equivalent to 
\begin{equation}\label{prob4}
\min_{i,j\in[n]} ~~  (y_i - y_j) (v_i - v_j).
\end{equation}
As we discuss below, the computation cost of the above problem can be reduced by making use of its structural properties. Let us denote $\B z_i = (y_i, v_i) \in \R^2$. Among the set of points $\{\B z_1,..., \B z_n\}$, we say $\B z_i$ is a ``left-top" point if for all $j\in [n]$, 
$$
\B z_j \notin \{  (u_1,u_2) \in \R^2 ~|~   u_1 \le y_i , ~ u_2 \ge v_i  \} \setminus \{\B z_i\} \ .
$$
We say $\B z_i$ is a ``right-bottom" point if for all $j\in [n]$, 
$$
\B z_j \notin \{  (u_1,u_2) \in \R^2 ~|~   u_1 \ge y_i , ~ u_2 \le v_i  \} \setminus \{\B z_i\} \ .
$$
Figure \ref{fig: scatter-z} shows an example of left-top and right-bottom points for a collection of $\B{z}_{i}$'s with noisy $\B y$. It can be seen that the number of left-top and right-bottom points can be much smaller than the total number of points. 
\begin{figure}[h!]
	\centering
	\includegraphics[height=5cm,width=7cm]{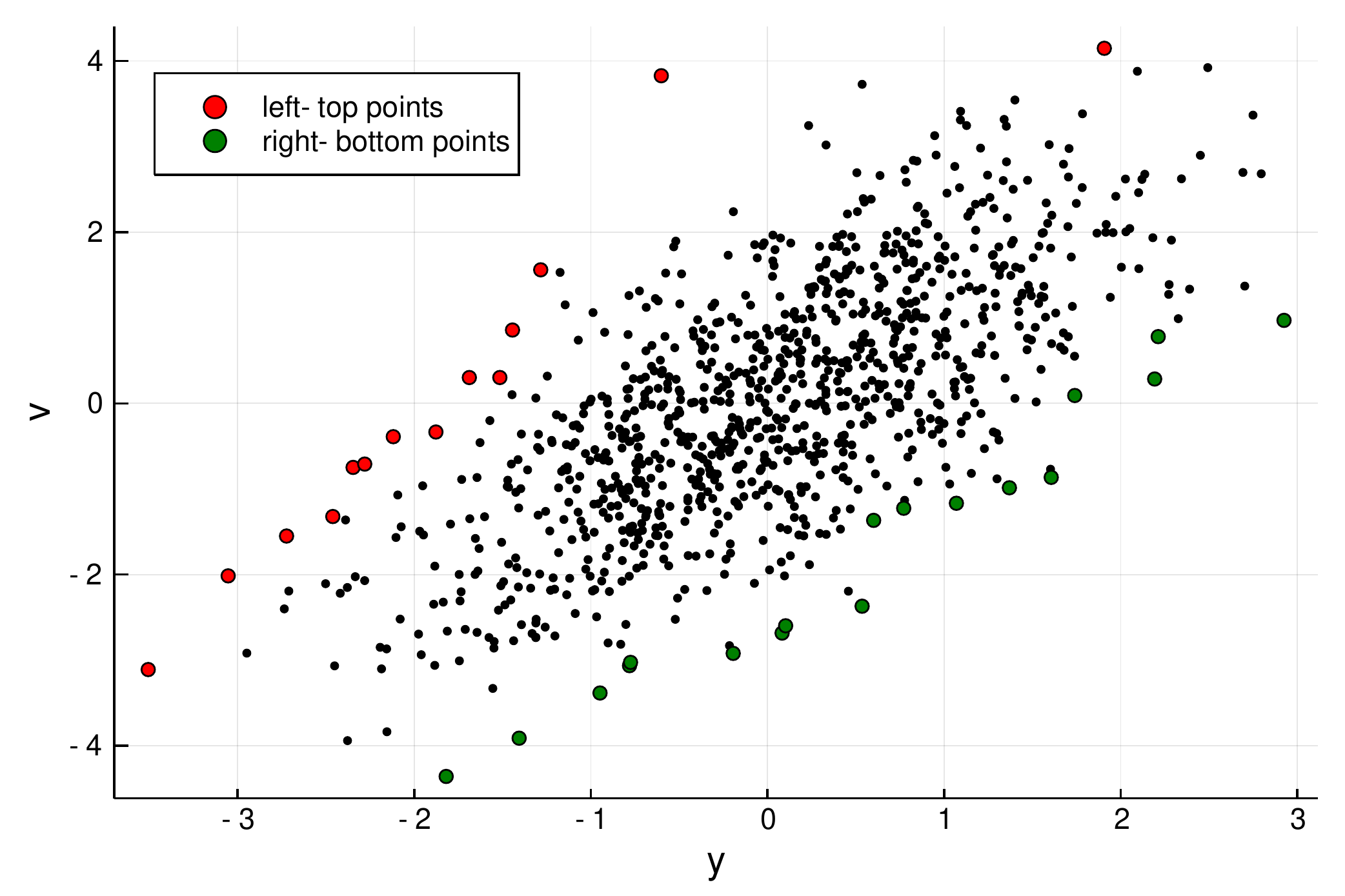}
	\label{fig: scatter-z}
	\caption{Figure illustrating the set Left-top points and right-bottom points in a 2D collection of points $\B z_i = (y_i, v_i) \in \R^2 $.}
\end{figure}

Let $(i^*,j^*)$ be an optimal solution to~\eqref{prob4}, then it must hold that one of $\{\B z_{i^*}, \B z_{j^*}\}$ is a left-top point and the other is a right-bottom point. Let $\cZ_{lt}$ and $\cZ_{rb}$ be the set of left-top and right-bottom points respectively, and define
$$
\cS_{lt} := \{i\in [n] ~| ~ \B z_{i} \in \cZ_{lt}\},~~~ 
\cS_{rb} := \{i\in [n] ~| ~ \B z_{i} \in \cZ_{rb}\} \ . 
$$
Then Problem~\eqref{prob4} is equivalent to 
\begin{equation}\label{prob5}
\min_{i\in\cS_{lt}, ~ j\in \cS_{rb}} ~~  (y_i - y_j) (v_i - v_j) 
\end{equation}
implying that it suffices to compute  values of $(y_i - y_j) (v_i - v_j)$ for $i\in\cS_{lt}$ and $j\in \cS_{rb}$. Algorithm~\ref{alg: fast swap} discusses how to compute $ \cS_{lt}$ and $\cS_{rb}$---this requires (a)~performing a sorting operation on $\B y$, which can be done once with a cost of $O(n\log(n))$; and 
(b)~two additional passes over the data with cost $O(n)$ (to be performed at every iteration of Algorithm~\ref{alg: AFW}).

The computation of \eqref{prob5} can be further simplified as discussed in the following section. 
\begin{algorithm}[h!]
	\caption{Fast algorithm for Problem~\eqref{prob4}}
	\label{alg: fast swap}
	\begin{algorithmic}
		\STATE \textbf{Input}: Vectors $\B v, \B y \in \R^n$. 
		
		\STATE \textbf{Step 1}: (Sorting) Find a permutation $\pi$ such that $y_{\pi(1)} \le y_{\pi(2)} \le \cdots \le y_{\pi(n)}$. Define $\bar y_i = y_{\pi(i)}$ and $\bar v_i = v_{\pi (i)}$ for all $i\in [n]$. 
		
		\STATE \textbf{Step 2}: (Construct $\cS_{lt}$) Initialize $\cS_{lt} = \{\pi(1)\}$ and $T = \bar v_{1}$. 
		\STATE \qquad \textbf{For} $i = 2,3,...,n$:
		\STATE \qquad\quad \textbf{If} $\bar v_i > T$: ~ $\cS_{lt} = \cS_{lt} \cup \{\pi(i)\}$; ~$T = \bar v_i$ ~~\textbf{end if}
		\STATE \qquad \textbf{end for}
		
		\STATE \textbf{Step 3}: (Construct $\cS_{rb}$) Initialize $\cS_{rb} = \{\pi (n)\}$ and $B = \bar v_n$. 
		\STATE \qquad \textbf{For} $i = n-1, n-2, ...., 1$:
		\STATE \qquad\quad \textbf{If} $\bar v_i < B$: ~ $\cS_{rb} = \cS_{rb} \cup \{\pi(i)\}$; ~$B = \bar v_i$ ~~\textbf{end if}
		\STATE \qquad \textbf{end for}
		
		\STATE \textbf{Step 4}: 
		Solve \eqref{prob5} using the methods discussed in Section \ref{sec:faster-compute-1}
		
	\end{algorithmic}
\end{algorithm}
\subsection{Faster computation of Problem~\eqref{prob5}}\label{sec:faster-compute-1} 
To simplify the computation of \eqref{prob5},
we introduce a partial order `$\preceq$' on the points in $\R^2$: For $\B p,\B q \in \R^2$, denote $\B p\preceq \B q$ if $p_1 \le q_1$ and $p_2 \le q_2$.
It is easy to check that for any two points $\B z_i, \B z_j \in \cZ_{rb}$, either it holds $\B z_i \preceq \B z_j$, or it holds $\B z_j \preceq \B z_i$. So we can write $\cZ_{rb}=\{\B z_{i_1}, ...., \B z_{i_{L'}}\}$ with
\begin{equation}\label{ordered-Zlt}
\B z_{i_1} \preceq \B z_{i_2 } \preceq \cdots \preceq \B z_{i_{L'}}. 
\end{equation}
For any $\B z_{m} \in \cZ_{lt}$, two cases can happen:
\begin{enumerate}
	\item There is no point $\B z_{i_t} \in \cZ_{rb}$ satisfying $(y_{i_t} - y_m) (v_{i_t} - v_m) \le 0$. 
	
	\item  There exist $\bar t ,\bar b \in [L']$ with $ \bar b \le \bar t$ such that 
	$ (y_m - y_{i_t}) (v_m - v_{i_t}) \le 0 $ for all $\bar b \le t \le \bar t$, and $ (y_m - y_{i_t}) (v_m - v_{i_t}) > 0 $ for all $t>\bar t$ or $t<\bar b$. 
\end{enumerate}
Because $\cZ_{rb}$ is nicely ordered as in \eqref{ordered-Zlt}, Case 1 above can be identified by a bisection method with
cost (at most) $ O(\log n)$. Similarly, for Case 2, the values of $\bar t$ and $\bar b$ can be found using bisection. 
Since the optimal value of \eqref{prob4} must be non-positive, 
we can compute $ (y_m - y_{i_t}) (v_m - v_{i_t}) $ only for $\bar b \le t \le \bar t$.  

The methods described for solving \eqref{prob4} are summarized in Algorithm \ref{alg: fast swap}. 
Finally, note that when $\dist(\B P^{(k)}, \B I_n) \ge R-1$, similar ideas are still applicable. When $\dist(\B P^{(k)}, \B I_n) = R-1$, we consider the following problem:
\begin{equation}\label{approximate2}
\min_{i,j} ~ (y_i - y_j) (v_{i} - v_j)~~~~ \text{s.t.} ~ i\in [n], ~ j\in \supp(\B P^{(k)}).
\end{equation}
Similarly, when $\dist(\B P^{(k)}, \B I_n) = R$, we consider:
\begin{equation}\label{approximate3}
\min_{i,j} ~ (y_i - y_j) (v_{i} - v_j)~~~~ \text{s.t.} ~  i, j\in \supp(\B P^{(k)}).
\end{equation}
Problems \eqref{approximate2} and \eqref{approximate3} can also be efficiently solved by finding the sets of left-top and right-bottom points and using the partial order to simplify the computation. We omit the details for brevity.

\section{Experiments}\label{sec:expts}

We perform numerical experiments to study the performance of Algorithm \ref{alg: AFW}. 

\smallskip
\noindent
\textbf{Data generation.} We consider the setup in our basic model~\eqref{model1},
where entries of $\B X\in \R^{n\times d}$ are iid $N(0,1)$; $\B \beta^*$ is generated uniformly from the unit sphere in $\R^d$ (i.e., $\|\B \beta^*\| = 1$), and $\B \beta^*$ is independent of $\B X$.  We consider two schemes for generating the permutation $\B P^*$: (a) {\emph{Random scheme}}: select $r$ coordinates uniformly from $\{1, \ldots, n\}$. (b) {\emph  {Equi-spaced scheme}}: 
Assume $y_1\le \cdots \le y_n$ (otherwise re-order the data).
	Let $a_1<\cdots <a_r$ be the sequence of $r$ equi-spaced real numbers with $a_1 = \min_{i\in [n]} y_i$ and $a_n = \max_{i\in [n]} y_i$. 
	Select $r$ indices $i_1<\cdots <i_r$ such that $i_1 = \argmin_{i\in [n] } |y_i - a_1|$ and $i_s = \argmin_{i_{s-1}+1 \le i \le n} |y_i  - a_s|$ for all $2\le s \le r$.  
After the $r$ coordinates are chosen, we generate a uniformly distributed random permutation on these $r$ coordinates.\footnote{Note that $\B P^*$ may not satisfy $ \dist(\B P^* , \B I_n) = r$, but $ \dist(\B P^* ,\B I_n)$ will be close to $r$.}

We generate $\B \ep$ (independent of $\B X$ and $\B \beta^*$) with  $\epsilon_{i} \stackrel{\text{iid}}{\sim} N(0,\sigma^2)$ for some $\sigma\ge 0$ ($\sigma = 0$ corresponds to the noiseless setting). 
{Unless otherwise specified, we set the tolerance $\mathfrak{e} = 0$ and $K=\infty$ in Algorithm \ref{alg: AFW}}.

\subsection{Experiments for the noiseless setting}

We first consider the noiseless setting ($\B \ep = \B 0$)
with different combinations of $(d,r,n)$. 
We use the \emph{random scheme} to generate the unknown permutation~$\B P^*$.
We set $R = n$ 
in 
Algorithm~\ref{alg: AFW} and a maximum iteration limit of $1000$. 
While our algorithm parameter choices are not covered by our theory, in practice when $r$ is small, our local search algorithm converges to optimality; and the number of iterations is bounded by a small constant multiple of $r$ (e.g., for $r=50$, the algorithm converges to optimality within around 60 iterations).

Figure \ref{figure: r_vs_error} presents preliminary results on examples with $n = 500$, $d\in\{20,50,100,200\}$, and 40 roughly equi-spaced values of $r \in [10,400]$.
In Figure \ref{figure: r_vs_error} [left panel], we plot the \text{Hamming distance} of the solution $\hatP$ computed by Algorithm \ref{alg: AFW} and the  underlying permutation $\B P^*$ (i.e. $\dist(\hatP, \B P^*)$) versus $r$.
In Figure \ref{figure: r_vs_error} [right panel], we present errors in estimating $\B{\beta}^*$ versus $r$. More precisely, let $ \hatbeta $ be the solution computed by Algorithm \ref{alg: AFW} (i.e. $\hatbeta = (\B X^\T \B X)^{-1}\B X^\T \hatP \B y$), then the \textit{beta error} is defined as $ \| \hatbeta - \B \beta^* \|/\|\B \beta^* \| $. For each choice of $(r,d)$, we consider the average over $50$ independent replications (the vertical bars show standard errors, which are hardly visible in the figures). 
As shown in Figure \ref{figure: r_vs_error}, when $r$ is small, the underlying permutation $\B P^*$ can be exactly recovered, and thus the corresponding \text{beta error} is also $0$. As $r$ becomes larger, Algorithm \ref{alg: AFW} fails to recover $\B P^*$ exactly; and   $\dist(\B P^*, \hatP)$ is close to the maximal possible value $n = 500$.
In contrast, the estimation error appears to vary more smoothly: As the value of $r$ increases, \emph{beta error} increases. We also observe that the recovery of $\B{P}^*$ depends upon the number of covariates $d$ --- permutation recovery performance deteriorates with increasing $d$.
This is consistent with our theory suggesting that the performance of our algorithm depends upon both $r$ and $d$.

\begin{figure}[h]
	\begin{minipage}[t]{0.48\linewidth} 
		\centering 
		\includegraphics[height=4.5cm,width=6.3cm]{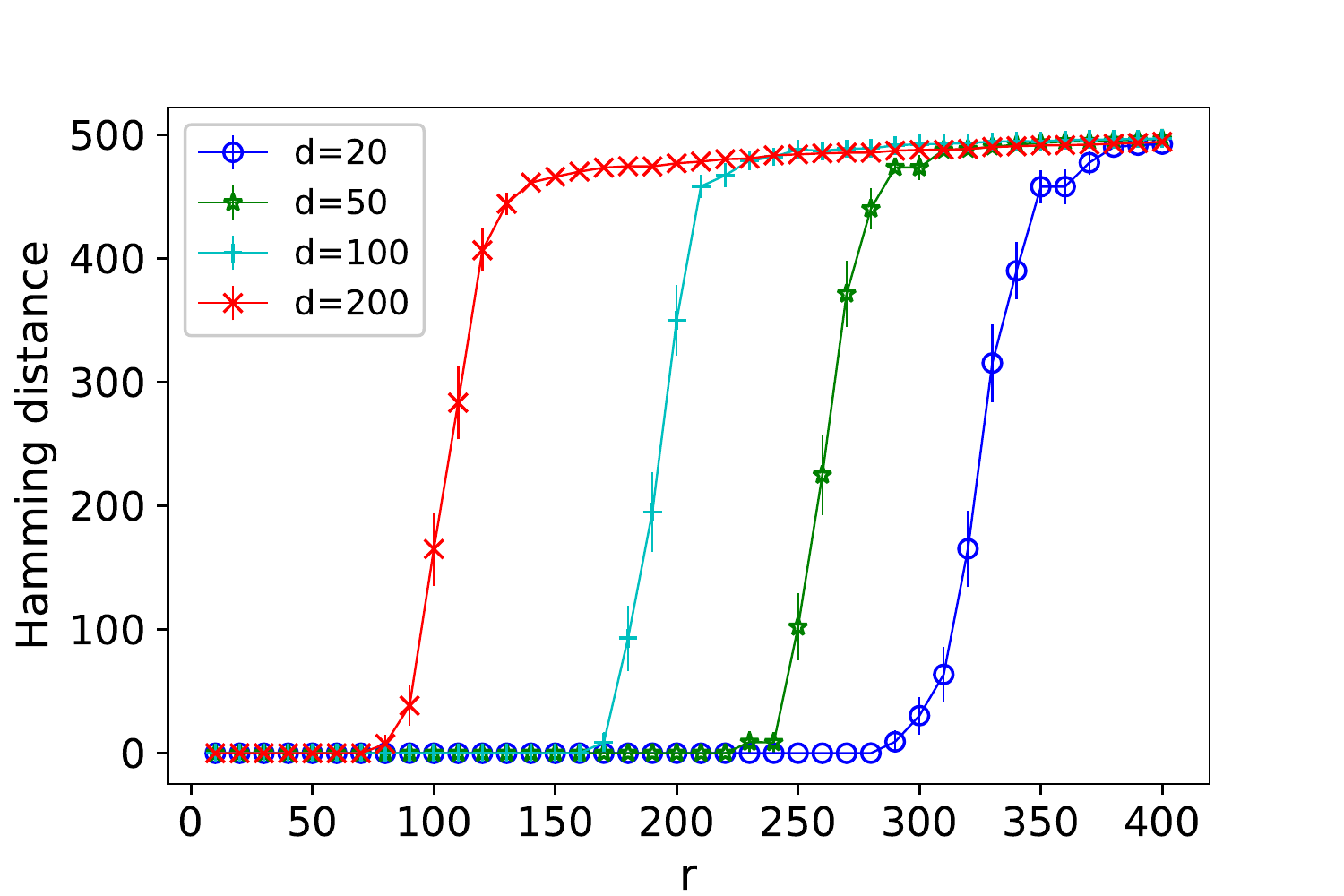} 
	\end{minipage}
	\begin{minipage}[t]{0.48\linewidth} 
		\centering 
		\includegraphics[height=4.5cm,width=6.3cm]{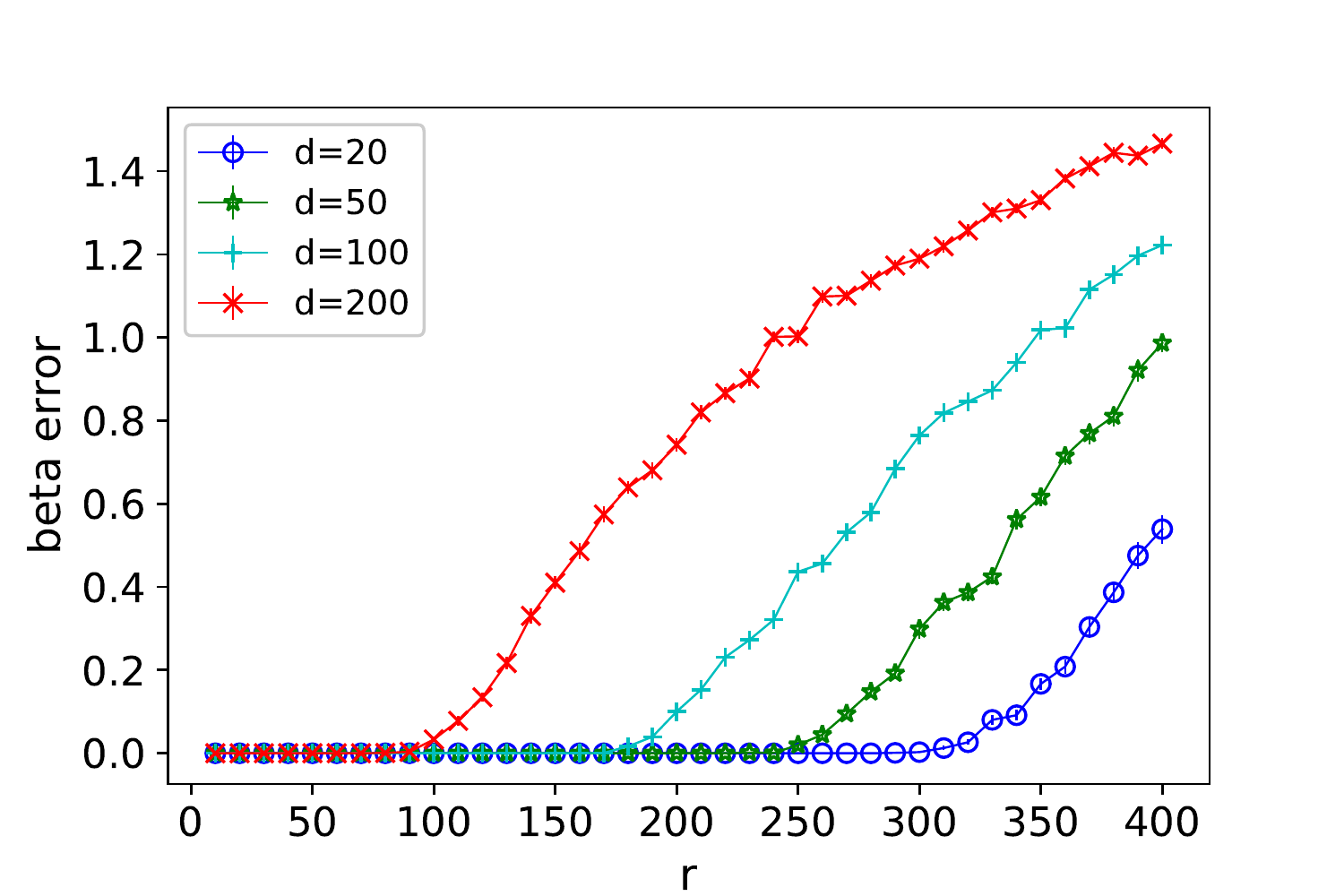}
	\end{minipage} 
	\vspace{-1em}
	\caption{Left: Values of Hamming distance $\dist(\hatP, \B P^*)$ versus $r$. 
		Right: Values of beta error $\| \hatbeta - \B \beta^* \| / \| \B \beta^*\|$ versus $r$.} 
	\label{figure: r_vs_error}
\end{figure}

\subsection{Experiments for the noisy setting}
We explore the performance of Algorithm~\ref{alg: AFW} under the noisy setting ($\B \ep \neq \B 0$).

\smallskip

\noindent \textbf{Performance for different values of $R$:} We denote \textit{Relative Obj} as the objective value computed by Algorithm \ref{alg: AFW} divided by $\| \tdH \B\ep\|^2$. 
Figure \ref{fig: early stopping} presents the Relative Obj, beta error and Hamming distance of the local search algorithm with different values of $R$ (x-axis corresponds to the values of $R$). 
Here we consider $n=1000$, $d=10$, $r=10$ and $\sigma=0.1$; and use the {\emph  {equi-spaced scheme}} to choose the mismatched coordinates in $\B P^*$.
We highlight the value at $R = r = 10$ by a red point. 
As shown in Figure \ref{fig: early stopping}, as $R$ increases, 
the Relative Obj decreases below $1$ -- this is consistent with our theory stating that with a proper choice of $R$, the final objective value will be below a constant multiple of $\|\tdH \B\ep\|^2$. 

As $R$ increases, different from the Relative Obj profile,
the \text{beta error} and \text{Hamming distance} first decrease then increase. 
This appears to suggest that when $R$ is too large, Algorithm~\ref{alg: AFW} can overfit and further regularization may be necessary to mitigate overfitting. A detailed investigation of this matter is left as future work. 
In this example, the best beta error and Hamming distance are achieved when $R$ equals $r$. 
Note that in Figure \ref{fig: early stopping} [left panel], the \text{Relative Obj} is close to $1$ when we choose $R $ close to $r$. 
Therefore, if we have a good estimate of the noise level $\sigma$ (but the exact value of $r$ is not available), we can choose a value of $R$ at which the Relative Obj is approximately $1$. 

Finally, we note that in the noisy case, the local search method cannot exactly recover $\B P^*$. Indeed, in the noisy case, if a solution to \eqref{intro: problem0} has to exactly recover $\B P^*$, we need to take a smaller value of $\sigma$ (see discussions in \cite{pananjady2017linear}). 
Even though in our example, we cannot exactly recover $\B P^*$, we may still be able to obtain a good estimate for
$\B \beta^*$---see Figure \ref{fig: early stopping} [middle panel].

\begin{figure}[]
	\resizebox{\textwidth}{!}{
		\begin{minipage}[t]{0.3\linewidth} 
			\centering 
			\includegraphics[height=3cm,width=3.9cm]{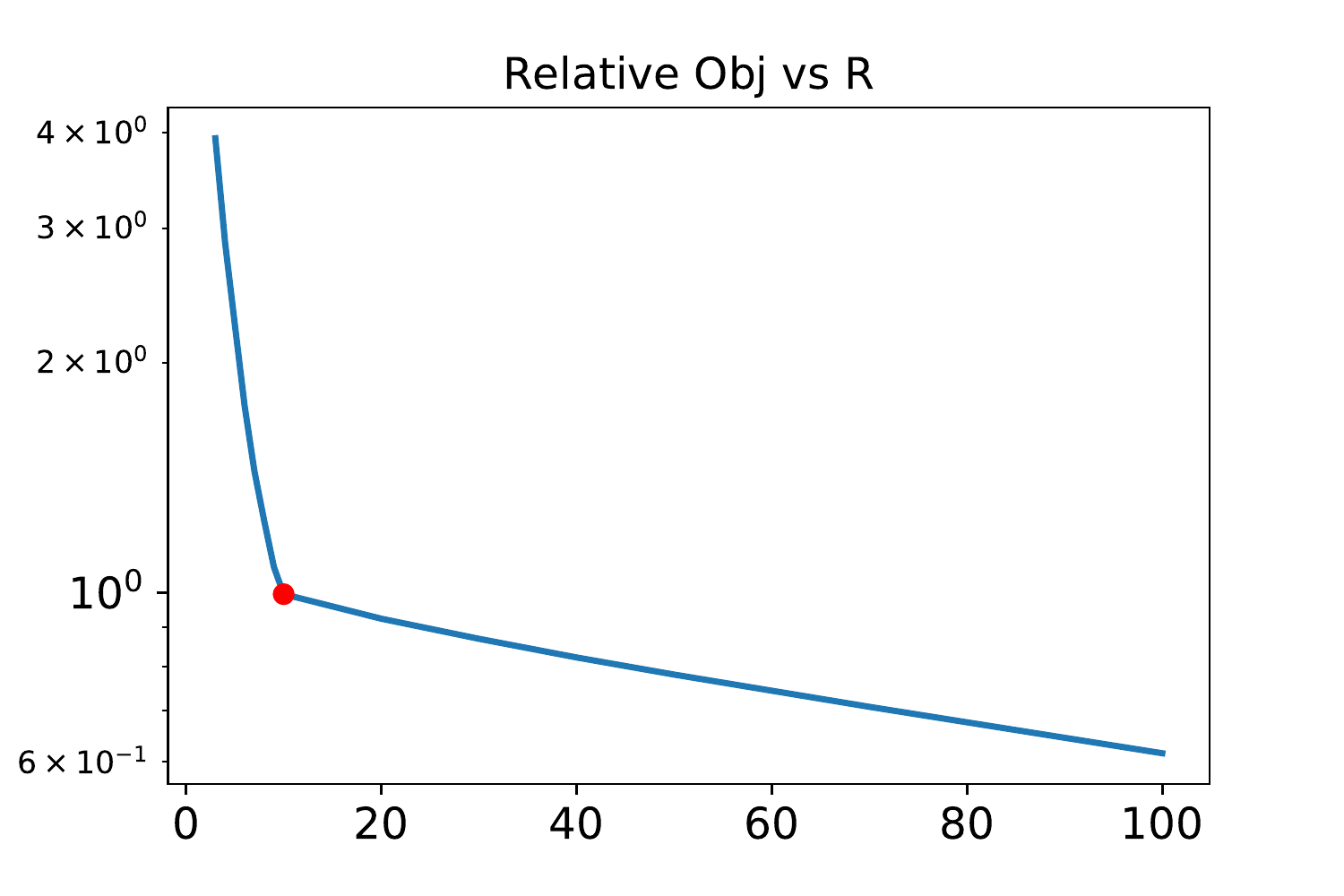} 
		\end{minipage}
		\begin{minipage}[t]{0.3\linewidth} 
			\centering 
			\includegraphics[height=3cm,width=3.9cm]{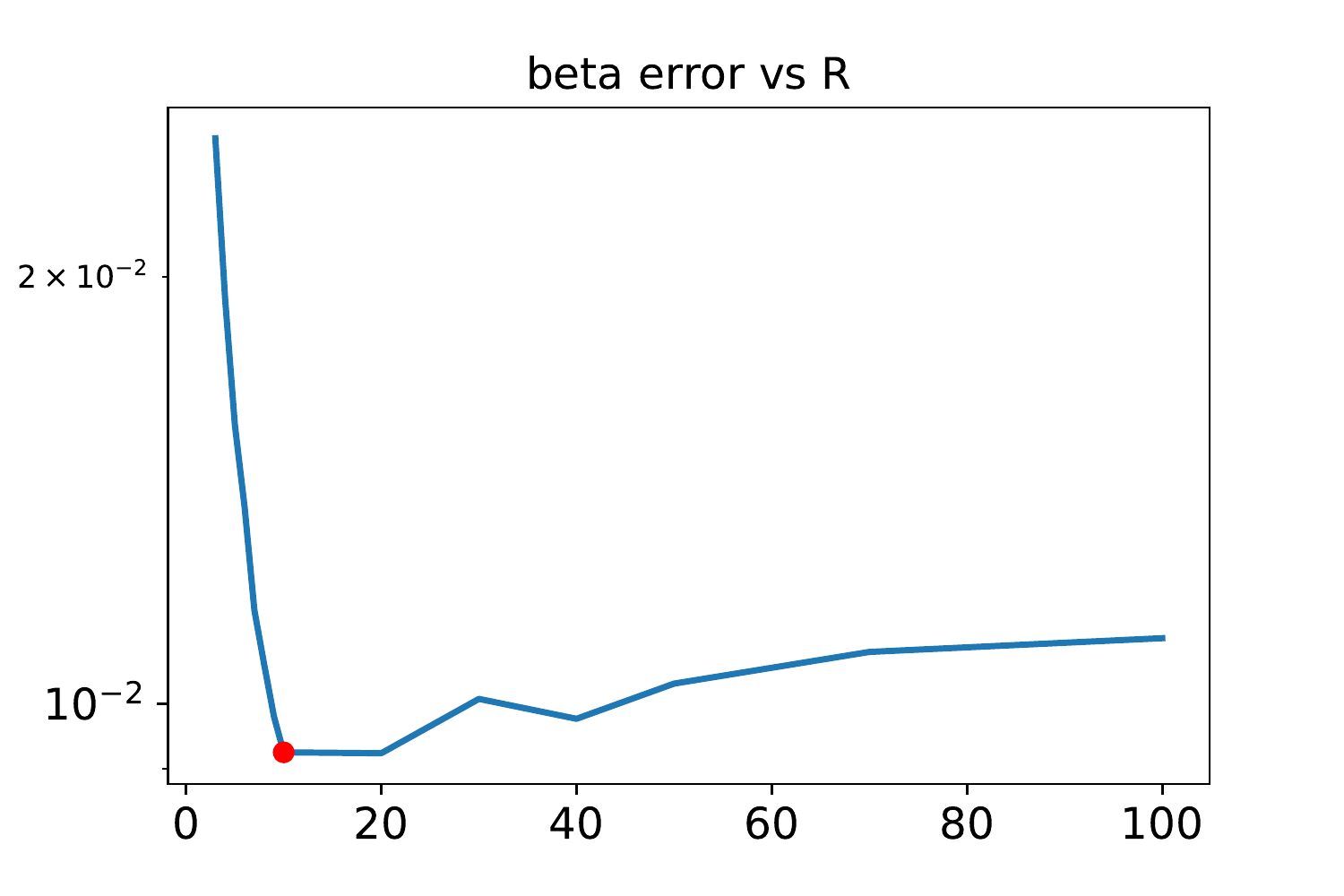} 
		\end{minipage} 
		\begin{minipage}[t]{0.3\linewidth} 
			\centering 
			\includegraphics[height=3cm,width=3.9cm]{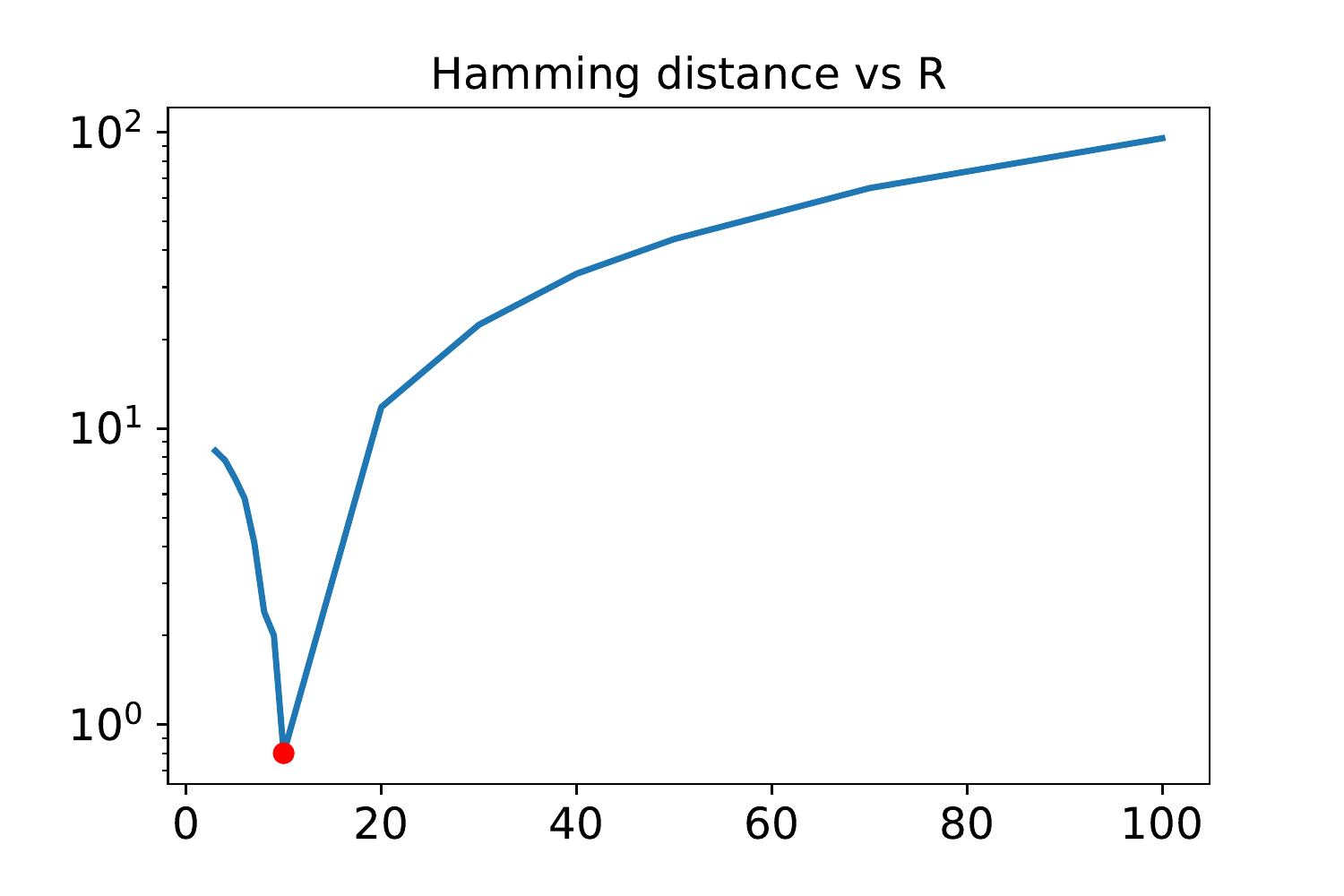} 
		\end{minipage} 
		\vspace{-1em}
	}
	\caption{Experiment on an instance with $n=1000$, $d=10$, $r=10$ and $\sigma=0.1$. 
		Left: Relative Obj vs $R$. ~Middle: beta error vs $R$. ~Right: Hamming distance vs $R$. The circled red point corresponds to $R = r$. 
	}
	\label{fig: early stopping}
\end{figure}

\smallskip

	\noindent \textbf{Estimating $\B{P}^*, \B\beta^*$ under different noise levels:}
	For a given $\sigma$ (standard deviation of the noise), let \textit{relative beta error} be the value $\| \hat{\B\beta} - \B\beta^* \| / ( \sigma\| \B\beta^*\|)$, 
	where $\hat{\B P}$ and $\hat{\B \beta}$ are the estimates available from 
	Algorithm~\ref{alg: AFW} upon termination. 
	
	Consider an example with $n = 500$ and $d=r=10$ and different values of $\sigma\in \{ 0.01, 0.03, 0.1, 0.3, 1.0\}$, 
	and use the \textit{random scheme} to generate the unknown permutation $\B P^*$. 
	We run Algorithm \ref{alg: AFW} with the setting $R = r$. 
	Figure \ref{fig: different noise} presents the values of Hamming distance and relative beta error under different noise levels. 
	In Figure \ref{fig: different noise} [left panel], it can be seen that as $\sigma$ increases, the Hamming distance also increases, and recovering $\B P^*$ becomes harder as the noise level becomes larger. In Figure \ref{fig: different noise} [right panel], we can see that the relative beta error almost does not change under different values of $\sigma$. This appears to be consistent with our conclusion in Theorem \ref{theorem: consistency} that $\| \hat{\B\beta} - \B \beta^* \|$ will be bounded by a value proportional to $\sigma$. 
	
	\begin{figure}[]
		\resizebox{\textwidth}{!}{
			\begin{minipage}[t]{0.5\linewidth} 
				\centering 
				\includegraphics[height=4cm,width=5.9cm]{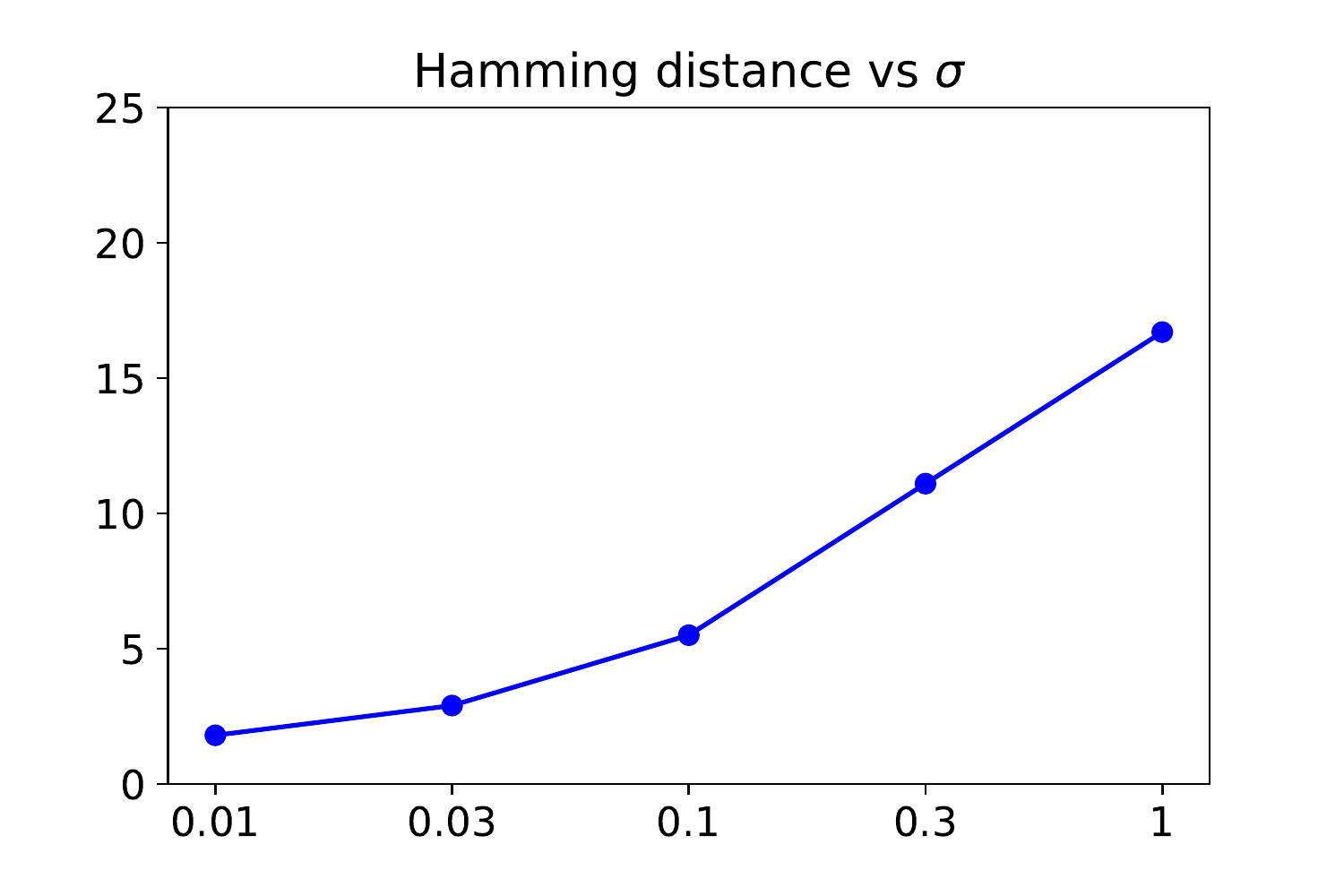} 
			\end{minipage}
			\begin{minipage}[t]{0.5\linewidth} 
				\centering 
				\includegraphics[height=4cm,width=5.9cm]{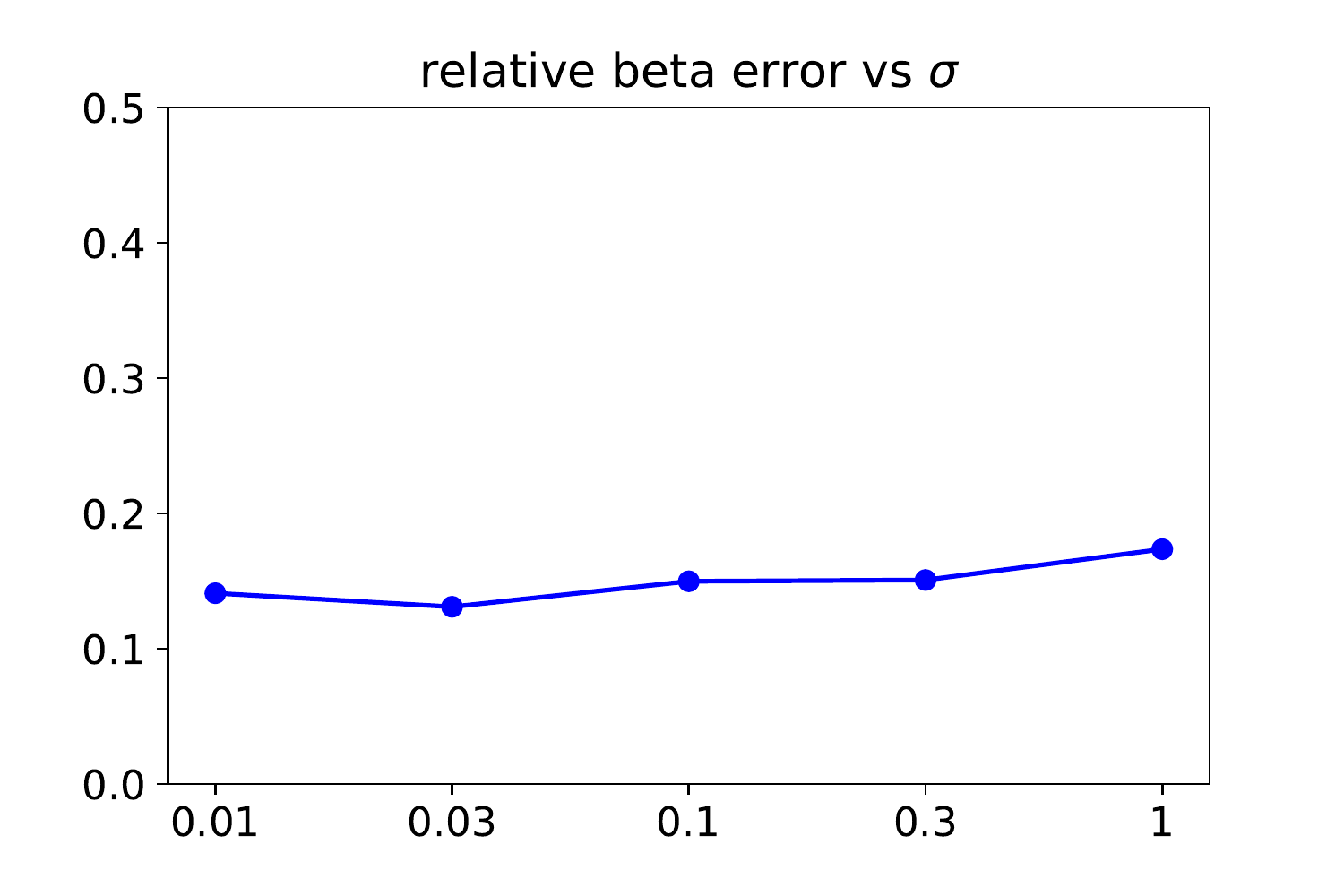} 
			\end{minipage} 
		}
		\caption{Experiment on an instance with $n=500$, $d=10$, $r=10$, and different noise levels $\sigma\in \{0.01, 0.03, 0.1, 0.3, 1.0\}$. Left: Hamming distance vs $\sigma$. Right: relative beta error vs $\sigma$.
		}
		\label{fig: different noise}
	\end{figure}

\subsection{Comparisons with existing methods}
We compare across the following methods for~\eqref{intro: problem1}:
\begin{itemize}
	\item \texttt{AltMin}: The alternating minimization method of~\cite{haghighatshoar2017signal}. We initialize with $\B P =\B  I_n$ and $\B \beta = \B 0$, and
	alternately minimize over $\B P$ and $\B \beta$  until no improvement can be made.
	\item \texttt{StoEM}: The stochastic expectation maximization  method \cite{abid2018stochastic}. We run the algorithm for 30 steps under the default setting. 
	\item \texttt{S-BD}:  
	The robust regression relaxation method of~\cite{slawski2019linear}. 
	We set the regularization parameter $\lam = 4(1+M)\sigma \sqrt{2\log(n)/n}$ with $M=1$ (this value of $\lam$ is given by Theorem 2 of \cite{slawski2019linear}). 
\end{itemize}
For both \texttt{AltMin} and \texttt{StoEM}, we use the implementation provided in the repository\footnote{\url{https://github.com/abidlabs/stochastic-em-shuffled-regression}} accompanying paper~\cite{abid2018stochastic}.
We compare these methods with Algorithm~\ref{alg: AFW} (denoted by \texttt{Alg-1}) and the variant of Algorithm~\ref{alg: AFW} with the fast approximate local search steps introduced in Section~\ref{section: fast local search steps} (denoted by \texttt{Fast-Alg-1}). 
We run \texttt{Alg-1} and \texttt{Fast-Alg-1} by setting $R = r$.

Table~\ref{table: compare} presents the \text{beta errors} of different methods on an example with $n = 500$, $d=10$; and $\B P^* $ chosen by the \emph{random scheme}. The presented values are the average of $10$ independent replications. 

As shown in Table~\ref{table: compare}, in the noiseless setting, \texttt{Alg-1}
can recover the true value of $\B \beta^*$ for $r$ up to $300$.
\texttt{Fast-Alg-1} is quite similar to \texttt{Alg-1}, though the performance is marginally worse for larger values of $r$.
In contrast, \texttt{AltMin} and \texttt{StoEM} are not able to exactly estimate $\B \beta^*$ even for small values of $r$, and for all values of $r$ they have large values of beta error. \texttt{S-BD} which is based on convex optimization, can also recover $\B \beta^*$ for $r\le 200$, but for $r = 300$ its performance degrades and has a large beta error. 

In the noisy setting, with $\sigma = 0.1$, \texttt{Alg-1} and \texttt{Fast-Alg-1} have similar performance, and compute a value of $\B \beta$ with much smaller beta error compared to \texttt{AltMin} and \texttt{StoEM}. For small values of $r$ ($\le 100$), \texttt{S-BD} has a similar performance to \texttt{Alg-1} and \texttt{Fast-Alg-1}, while for $r = 200$ and $300$, its performance degrades and has a much larger beta error than \texttt{Alg-1} and \texttt{Fast-Alg-1}. 

\begin{table}[]
	\centering
	\begin{tabular}{|cc|ccccc|}
		\hline
		&     & \multicolumn{5}{c|}{beta error}                                                 \\ \hline
		\multicolumn{1}{|c|}{$\sigma$}             & $r$   & \texttt{Alg-1} & \texttt{Fast-Alg-1} & \texttt{AltMin} & \texttt{StoEM} & \texttt{S-BD} \\ \hline
		\multicolumn{1}{|c|}{\multirow{4}{*}{0}}   & 50  & 0.000       & 0.000        & 0.001           & 0.154          & 0.000           \\
		\multicolumn{1}{|c|}{}                     & 100 & 0.000       & 0.000        & 0.027           & 0.483          & 0.000           \\
		\multicolumn{1}{|c|}{}                     & 200 & 0.000       & 0.001        & 0.121           & 0.884          & 0.000           \\
		\multicolumn{1}{|c|}{}                     & 300 & 0.000       & 0.001        & 0.190           & 0.944          & 0.123           \\ \hline
		\multicolumn{1}{|c|}{\multirow{4}{*}{0.1}} & 50  & 0.018       & 0.018        & 0.061           & 0.207          & 0.021           \\
		\multicolumn{1}{|c|}{}                     & 100 & 0.021       & 0.021        & 0.080           & 0.421          & 0.029           \\
		\multicolumn{1}{|c|}{}                     & 200 & 0.019       & 0.022        & 0.116           & 0.881          & 0.063           \\
		\multicolumn{1}{|c|}{}                     & 300 & 0.047       & 0.052        & 0.213           & 0.942          & 0.282           \\ \hline
	\end{tabular}
	\caption{Comparison with existing methods on an example with $n=500$, $d=10$}
	\label{table: compare}
\end{table}

\subsection{Scalability to large instances}
We explore the scalability of our proposed approach to large $n$ problems (from $n\approx 10^4$ up to $n\approx 10^7$)---for these instances, \texttt{Fast-Alg-1} appears to be computationally attractive. 
All these experiments are run on the MIT engaging cluster with 1 CPU and 16GB memory. The codes are written in Julia 1.2.0. 

We consider examples with $d = 100$, $r = 50$ and $n\in \{10^4, 10^5,$ $ 10^6, 10^7\}$. 
Here, the mismatched coordinates of $\B P^*$ are chosen based on the \emph{random scheme}. We set $R = r$ for all instances. For these examples, we do not form
the $n \times n$ matrices $\tdH$ or $\B H$ explicitly, but compute a thin QR decomposition of $\B X$ ($\B Q\in \R^{n\times d}$) and maintain $\B Q $ in memory. The results are presented in Table~\ref{table: runtime}, where ``total time" is the total runtime of \texttt{Fast-Alg-1} upon termination, ``QR time" is the time used for the QR-decomposition, and ``iterations" are the number of iterations taken by the local search method till convergence. All numbers reported in the table are averaged across 10 independent replications. As shown in Table \ref{table: runtime}, \texttt{Fast-Alg-1} can solve examples with $n$ up to $10^7$ (and $d = 100$) within around 200 seconds --- this runtime (s) includes the  time to complete around $60$ iterations of local search steps and the time to do the QR decomposition.  The total runtime is empirically seen to be of the order $O(n)$ as $n$ increases. 
Note that the QR time (i.e., time to perform the QR decomposition) can be viewed as a benchmark runtime for ordinary least squares. Hence,  for the examples considered, 
the runtime of \texttt{Fast-Alg-1} appears to be a constant multiple of the runtime of performing ordinary least squares (the total time will increase with $r$ and/or $R$).
Interestingly, it can be seen that the runtimes for the noisy case ($\sigma = 0.1$) are smaller than the noiseless case ($\sigma = 0$). We believe this is because Algorithm~\ref{alg: fast swap} is faster for the noisy case.
In particular, for the noisy case, we empirically observe the number of ``left-top" points and ``right-bottom" points to be fewer than those in the noiseless case.

\begin{table}[h]
	\centering
	\begin{tabular}{|c|ccc|ccc|}
		\hline
		& \multicolumn{3}{c|}{$\sigma = 0$} & \multicolumn{3}{c|}{$\sigma = 0.1$} \\ \hline 
		n               & total time(s) & QR time(s) & iterations & total time(s)  & QR time(s)  & iterations \\ \hline 
		$10^4$          &    0.2        &    0.1        &   57.8      &    0.2         &     0.1        &    58.3     \\ 
		$10^5$ &    2.2        &    0.7        &   58.9      &    1.5         &     0.7        &    58.8     \\ 
		$10^6$          &    21.3        &    6.3        &   58.7      &    13.9         &     6.4        &    56.4     \\ 
		$10^7$ &    212.6        &    64.5        &   58.5      &    152.9         &     64.8        &    60.7     \\ \hline 
	\end{tabular}
	\caption{Runtimes of \texttt{Fast-Alg-1} on instances with $d = 100$, $r=50$ and different $n$. For reference, the time taken by \texttt{Alg-1} in the case $n=10^4$ is 100 (s), which is 500$X$ slower than \texttt{Fast-Alg-1}.}
	\label{table: runtime}
\end{table}

\section{Acknowledgments} The authors would like to thank the anonymous referees for their constructive comments that led to an improvement of the paper.

\renewcommand{\theequation}{A.\arabic{equation}}
\appendix

\section{Proofs and technical results}

\subsection{Technical lemmas}
\begin{lemma}\label{lemma: sample covariance estimation}
	(Covariance estimation)
	Let $\B X =  [\B x_1,...,\B x_n]^\T $ be a random matrix in $ \R^{n\times d} $. Suppose rows $\B x_1,...,\B x_n$ are iid zero-mean random vectors in $\R^d$ with covariance matrix $\B \Sigma\in \R^{d\times d}$. Suppose $\|\B x_i\|\le b$ almost surely. 
	Then for any $t >0$, it holds
	$$
	\Pb \lt(   \itl  \frac{1}{n} \B X^\T \B X - \B \Sigma \itl_2 \ge t \itl \B \Sigma\itl_2 \rt) \le 2d \exp \lt( - \frac{n t^2 \itl \B \Sigma \itl_2 }{2b^2(1+t)} \rt). 
	$$
\end{lemma}
See e.g. Corollary 6.20 of \cite{wainwright2019high} for a proof.

\begin{lemma}\label{lemma: supp-inclusion}
	Suppose three permutation matrices $\B P, \tdP,  \B Q \in \Pi_n$ satisfy 
	$$
	\supp(\tdP \B Q^{-1}) \subseteq \supp(\B P \B Q^{-1}) \quad {\rm and} \quad \supp(\B Q) \subseteq \supp(\B P) \ ,
	$$
	then it holds $\supp(\tdP) \subseteq \supp( \B P) $.
\end{lemma}

\begin{proof}
	We just need to show that for any $i\notin \supp(\B P)$, it holds $  i\notin \supp(\tdP)$. Let $i\notin \supp(\B P)$, then $\B e_i^\T \B P = \B e_i^\T $. Since $ \supp(\B Q ) \subseteq \supp(\B P) $, we also have $ \B e_i^\T \B Q = \B e_i^\T  $. So it holds $  \B e_i^\T \B P \B Q^{-1} = \B e_i^\T $ or equivalently $ i \notin \supp(\B P \B Q^{-1} ) $. Because $ \supp(\tdP \B Q^{-1} )  \subseteq \supp(\B P \B Q^{-1} )$, we have $ i \notin \supp(\tdP \B Q^{-1} ) $, or equivalently $  \B e_i^\T \tdP \B Q^{-1} = \B e_i^\T $. This implies $ \B e_i^\T \tdP = \B e_i^\T \B Q = \B e_i^\T$, or equivalently, $ i \notin \supp(\tdP) $. 
\end{proof}

\begin{lemma}\label{lemma: step-bound}
	Suppose Assumption \ref{ass-main0} holds. 
	Let $\tdP ,\B P \in \cN_R(\B I_n) \subseteq \Pi_n$ with $\dist(\B P, \tdP) \le 4$. Let 
	$$
	\Dt(\tdP, \B P) := \| \tdP \B y - \B P^* \B y \|^2 - 	\| \B P \B y - \B P^* \B y \|^2 , \quad 
	\Dt_{\tdH} (\tdP, \B P) := 	\| \tdH \tdP \B y \|^2 - 	\| \tdH  \B P \B y \|^2 \ ,
	$$
	then it holds $| \Dt(\tdP, \B P) - \Dt_{\tdH} (\tdP, \B P)|
	< 
	L^2/5 $.
\end{lemma}

\begin{proof}
	Let $\B z := \tdP \B y - \B P\B y $. Note that
	\begin{equation}\label{eq-p1}
	\begin{aligned}
	\Dt(\tdP, \B P) &= \| \B z+ \B P \B y - \B P^* \B y \|^2 - 	\| \B P \B y - \B P^* \B y \|^2 \\
	& = 	\| \B z\|^2 + 2 \la \B z, \B P \B y - \B P^* \B y \ra
	\end{aligned}
	\end{equation}
	and
	\begin{equation}\label{eq-p2}
	\begin{aligned}
	\Dt_{\tdH} (\tdP, \B P) &=
	\| \tdH \B z + \tdH \B P\B y \|^2 - \|\tdH \B P\B y \|^2 \\
	&=
	\| \tdH \B z \|^2 + 2 \la \tdH \B z , \tdH \B P \B y \ra \\
	&=
	\| \tdH \B z \|^2 + 2 \la \tdH \B z , \tdH(\B P\B y - \B P^*\B y ) \ra + 2 \la \tdH \B z , \tdH \B P^* \B y \ra \\
	&=
	\| \tdH \B z \|^2 + 2 \la \tdH \B z , \tdH(\B P\B y - \B P^*\B y ) \ra + 2 \la \tdH \B z , \tdH \B \ep \ra \ ,
	\end{aligned}
	\end{equation}
	where the last equality uses the fact $ \tdH \B P^* \B y = \tdH(\B X \B \beta^* + \B \ep) = \tdH \B \ep  $. 
	From \eqref{eq-p1} and \eqref{eq-p2} we have 
	\begin{eqnarray}
	\Dt(\tdP, \B P) - \Dt_{\tdH} (\tdP, \B P) &= &
	\| \B H \B z \|^2 + 2 \la \B  H \B z , \B  H( \B P \B y -\B  P^* \B y ) \ra - 2 \la \tdH \B z  , \tdH \B \ep \ra \nonumber\\
	&=&
	\| \B H \B z \|^2 + 2 \la  \B H \B z ,  \B H( \B P \B y - \B P^* \B y ) \ra + 2 \la  \B H \B z ,  \B H \B \ep \ra - 2 \la  \B z ,  \B \ep \ra \ , \nonumber
	\end{eqnarray}
	where the second equality is because $\tdH = \B I_n - \B H$. 
	As a result, 
	\begin{equation}\label{ineq-2}
	\begin{aligned}
	&| 	\Dt(\tdP, \B P) - \Dt_{\tdH} (\tdP, \B P)   | \\
	\le & 
	\| \B H \B z \|^2  + 2 \|\B H \B z\| \| \B H( \B P \B y - \B P^* \B y ) \| + 2 \|  \B H \B z\|   \|\B H \B \ep\| + 2 | \la \B  z ,  \B \ep \ra | .
	\end{aligned}
	\end{equation}
	Since $ \dist(\B P, \tdP) \le 4 $, we have $ \B z = \tdP \B y - \B P\B y \in \cB_4 $, hence by Assumption \ref{ass-main0} (3) and (1) we have 
	\begin{equation}\label{ineq-3}
	\| \B H \B z\|^2 \le \rho_n \| \B z \|^2 \le 4 \rho_n U^2  \ . 
	\end{equation}
	Since $ \dist (\B P, \B P^*) \le \dist (\B P, \B I_n) + \dist (\B P^*, \B I_n) \le R+r $, we have $ \B P\B y - \B P^* \B y \in \cB_{R+r} \subseteq \cB_{2R}$, hence by Assumption \ref{ass-main0} (3) and (1), \begin{equation}\label{ineq-4}
	\| \B H(\B P\B y - \B P^* \B y) \| \le \sqrt{R \rho_n} \| \B P\B y -\B  P^* \B y \|  \le \sqrt{ \rho_n} (R+r)U \le 2\sqrt{\rho_n} RU \ .
	\end{equation}
	Again because $ \B z \in \cB_4 $ we have 
	\begin{equation}\label{ineq-5}
	|\la \B z, \B \ep \ra | \le 4 U \| \B \ep\|_{\infty } \ .
	\end{equation}
	Combining \eqref{ineq-2} -- \eqref{ineq-5} we have 
	\begin{eqnarray}
	| 	\Dt(\tdP, \B P) - \Dt_{\tdH} (\tdP, \B P)   |  &\le& 
	4\rho_n U^2 + 4 \sqrt{\rho_n} U \cdot 2\sqrt{\rho_n} RU + 4 \sqrt{\rho_n} U \| \B H \B \ep \| + 8U \| \B \ep\|_{\infty} \nonumber\\
	&=&
	(4 + 8R)\rho_n U^2 + 4 \sqrt{\rho_n} U \| \B H \B \ep \| + 8U \| \B \ep\|_{\infty} \nonumber\\
	&<&
	9R\rho_n U^2 + 4 \sqrt{\rho_n} U \| \B H \B \ep \| + 8U \| \B \ep\|_{\infty} \ , \nonumber
	\end{eqnarray}
	where in the last inequality we use $R> 4$. 
	By Assumption \ref{ass-main0} (4) we know $\|\B \ep\|_{\infty} \le \bar\sigma $ and $\| \B H \B \ep \| \le \sqrt{d} \bar\sigma$, so we have 
	\begin{eqnarray}\label{part0}
	| 	\Dt(\tdP, \B P) - \Dt_{\tdH} (\tdP, \B P)   |  <  9 \rho_n R U^2 + 4 \sqrt{d \rho_n} U \bar\sigma + 8U \bar\sigma \ . 
	\end{eqnarray}
	From Assumption \ref{ass-main0} (4) we have $\bar \sigma \le \min\{0.5, (\rho_nd)^{-1/2}\} L^2/(80U) $. This implies 
	\begin{eqnarray}\label{part1}
	4 \sqrt{d \rho_n} U \bar\sigma + 8U \bar\sigma \le L^2/20 + L^2/20 \le L^2/10 \ . 
	\end{eqnarray}
	Note that by Assumption \ref{ass-main0} (3) we have $R\rho_n \le L^2/(90U^2)$, or equivalently, $9R\rho_nU^2 \le L^2/10$. Combining this with \eqref{part0} and \eqref{part1}, we have 
	$$
	\setlength{\belowdisplayskip}{3pt}
	| 	\Dt(\tdP, \B P) - \Dt_{\tdH} (\tdP, \B P)   |  < L^2 /10+ L^2/10 = L^2/5 \ . 
	$$
\end{proof}

\begin{lemma}\label{lemma: Non-expansiveness of infinity norm}
	(Decrease in infinity norm) Let $\B P,\tdP \in \Pi_n$ with $\dist(\B P,\tdP) = 2$. For any $\B v\in \R^n$, if $ \| \tdP \B v - \B v\|^2 <  \|  \B P \B v - \B v\|^2  $, then it holds $ \| \tdP \B v - \B v\|_{\infty} \le  \|  \B P \B v - \B v\|_{\infty}$.
\end{lemma}
\begin{proof}
	Let $i\in [n]$ be the index such that $|\B e_i^\T (\tdP \B v - \B v)| = \| \tdP \B v - \B v \|_{\infty}$. If $ \B e_i^\T \B P = \B  e_i^\T \tdP$, then immediately we have 
	$$
	\|  \B P \B v - \B v\|_{\infty} \ge |\B e_i^\T ( \B P \B v - \B v)| = |\B e_i^\T (\tdP \B v - \B v)| = \| \tdP \B v - \B v \|_{\infty} \ . 
	$$
	If $ \B e_i^\T \B P \neq  \B  e_i^\T \tdP$, assume $\ell\in [n]$ is the index such that 
	$ \B e_i^\T \tdP = \B e_{\ell}^\T \B P $. Since $\dist(\B P,\tdP) = 2$, it holds 
	$ \B e^\T_{\ell} \tdP = \B e_i^\T \B P $. Denote $ i_+:= \pi_{\B P}(i) $ and $\ell_+:= \pi_{\B P}(\ell)$. Because $ \B e_{i_+}^\T = \B e_i^\T P = \B e_\ell^\T \tdP $ and $\B e_{\ell_+}^\T = \B e_\ell^\T P = \B e_i^\T \tdP$, it holds $ i_+ = \pi_{\tdP}(\ell)  $ and $\ell_+ = \pi_{\tdP}(i) $. As a result, we have 
	$$
	\| \tdP \B v - \B v \|^2 - \|  \B P \B v - \B v \|^2 = (v_{\ell} - v_{i_+})^2 + (v_{i} - v_{\ell_+})^2 - (v_i - v_{i_+})^2 -  (v_\ell - v_{\ell_+})^2 \ .
	$$
	By the assumption that $ \| \tdP \B v - \B v\|^2 <  \|  \B P \B v - \B v\|^2  $, we have 
	\begin{equation}\label{ineq-1}
	(v_{\ell} - v_{i_+})^2 + (v_{i} - v_{\ell_+})^2 - (v_i - v_{i_+})^2 -  (v_\ell - v_{\ell_+})^2 < 0 \ . 
	\end{equation}
	Let us denote $L:= \| \tdP \B v-  \B v \|_{\infty} = |\B e_i^\T (\tdP \B v - \B v)| = |v_i - v_{\ell_+}|$. 
	In what follows, we discuss different cases depending upon the ordering of the values of $v_i$, $ v_{i_+} $, $v_{\ell}$ and $v_{\ell_+}$. Without loss of generality, we can assume $ v_{i_+} \ge v_i $. Then there are 12 cases of the ordering in $v_i$, $ v_{i_+} $, $v_{\ell}$ and $v_{\ell_+}$. In the following, the first 6 cases correspond to when $ v_{\ell_+} \ge v_{\ell} $, and the last 6 cases correspond to when $ v_{\ell_+} \le v_{\ell} $. 
	
	\noindent
	\textbf{Case 1}: $ v_{i_+} \ge v_{\ell_+} \ge v_i \ge v_{\ell} $. Then we have
	$\|  \B P \B v - \B v\|_{\infty} \ge |v_i - v_{i_+}| \ge |v_i - v_{\ell_+}| = L $.

	\noindent
	\textbf{Case 2}: $ v_{i_+}\ge v_i \ge v_{\ell_+}  \ge v_{\ell} $. Let $a = v_{i_+} - v_{i}$, $ b = v_i - v_{\ell_+}$ and $c= v_{\ell_+}  - v_{\ell}  $.
	Then
	\begin{equation}
	\begin{aligned}
	&(v_i - v_{\ell_+})^2 + (v_{\ell} - v_{i_+})^2 - 	(v_i - v_{i_+})^2 - 	(v_{\ell} - v_{\ell_+})^2 \\
	=&~b^2 + (a+b+c)^2 - a^2 -c^2 \ge 2b^2 \ge 2L^2>0 \ .
	\end{aligned}
	\nonumber
	\end{equation}
	This is a contradiction to \eqref{ineq-1}. So this case cannot appear. 
	
	\noindent
	\textbf{Case 3}: $ v_{i_+} \ge v_{\ell_+}  \ge v_{\ell} \ge v_i$. Then we have 
	$\|  \B P \B v - \B v\|_{\infty} \ge |v_i - v_{i_+}| \ge |v_i - v_{\ell_+}| = L $.
	
	\noindent
	\textbf{Case 4}: $ v_{\ell_+} \ge v_{i_+} \ge v_{\ell} \ge v_{i} $. Let $a =  v_{\ell_+} - v_{i_+} $, $b =  v_{i_+} - v_{\ell}$ and $ c = v_{\ell} - v_{i}$. Then 
	\begin{eqnarray}
	&& (v_i - v_{\ell_+})^2 + (v_{\ell} - v_{i_+})^2 - 	(v_i - v_{i_+})^2 - 	(v_{\ell} - v_{\ell_+})^2 \nonumber\\
	&=&
	(a+b+c)^2 + b^2 - (b+c)^2 - (a+b)^2 = 2ac \ge 0 \ . \nonumber
	\end{eqnarray}
	This is a contradiction to \eqref{ineq-1}. So this case cannot appear. 
	
	\noindent
	\textbf{Case 5}: $ v_{\ell_+}\ge v_{\ell}  \ge v_{i_+} \ge v_{i} $. Let $ a= v_{\ell_+}- v_{\ell}$, $b = v_{\ell}  - v_{i_+}$ and $c = v_{i_+} - v_{i} $. Then 
	\begin{eqnarray}
	&& (v_i - v_{\ell_+})^2 + (v_{\ell} - v_{i_+})^2 - 	(v_i - v_{i_+})^2 - 	(v_{\ell} - v_{\ell_+})^2 \nonumber\\
	&=&
	(a+b+c)^2 + b^2 - a^2 - c^2  \ge 0 \ . \nonumber
	\end{eqnarray}
	This is a contradiction to \eqref{ineq-1}. So this case cannot appear. 
	
	\noindent
	\textbf{Case 6}: $ v_{\ell_+}\ge v_{i_+} \ge v_{i}\ge v_{\ell}   $. Then 
	$\|  \B P \B v - \B v\|_{\infty} \ge |v_\ell - v_{\ell_+}| \ge |v_i - v_{\ell_+}| = L $.
	
	\noindent
	\textbf{Case 7}: $ v_{i_+}\ge v_{\ell} \ge v_{i}\ge v_{\ell_+}   $. Then
	$\|  \B P \B v - \B v\|_{\infty} \ge |v_\ell - v_{\ell_+}| \ge |v_i - v_{\ell_+}| = L $.
	
	\noindent
	\textbf{Case 8}: $ v_{i_+} \ge v_{i}   \ge v_{\ell} \ge v_{\ell_+}   $. Let $a = v_{i_+} - v_{i}$, $ b = v_{i}   - v_{\ell}  $, $ c = v_{\ell} - v_{\ell_+} $. Then 
	\begin{eqnarray}
	&& (v_i - v_{\ell_+})^2 + (v_{\ell} - v_{i_+})^2 - 	(v_i - v_{i_+})^2 - 	(v_{\ell} - v_{\ell_+})^2 \nonumber\\
	&=&
	(b+c)^2 + (a+b)^2 - a^2 - c^2 \ge 0 \ . \nonumber
	\end{eqnarray}
	This is a contradiction to \eqref{ineq-1}. So this case cannot appear.

	\noindent
	\textbf{Case 9}: $ v_{i_+}   \ge v_{\ell} \ge v_{\ell_+} \ge v_{i} $. Then 
	$\|  \B P \B v - \B v\|_{\infty}\ge |v_i - v_{i_+}| \ge |v_i - v_{\ell_+}| = L $.

	\noindent
	\textbf{Case 10}: 
	$ v_{\ell}   \ge v_{i_+} \ge v_{\ell_+} \ge v_{i} $. Then 
	$ \|  \B P \B v - \B v\|_{\infty} \ge |v_i - v_{i_+}| \ge |v_i - v_{\ell_+}| = L  $.

	\noindent
	\textbf{Case 11}: 
	$ v_{\ell} \ge v_{\ell_+}   \ge v_{i_+} \ge v_{i} $. Let $ a= v_{\ell} - v_{\ell_+}  $, $b = v_{\ell_+}   - v_{i_+}$ and $ v_{i_+} - v_{i} $. Then
	\begin{eqnarray}
	&& (v_i - v_{\ell_+})^2 + (v_{\ell} - v_{i_+})^2 - 	(v_i - v_{i_+})^2 - 	(v_{\ell} - v_{\ell_+})^2 \nonumber\\
	&=&
	(b+c)^2 + (a+b)^2 - a^2 - c^2 \ge 0 \ . \nonumber
	\end{eqnarray}
	This is a contradiction to \eqref{ineq-1}. So this case cannot appear.

	\noindent
	\textbf{Case 12}: 
	$ v_{\ell}   \ge v_{i_+} \ge v_{i} \ge v_{\ell_+}  $. 
	Then $\|  \B P \B v - \B v\|_{\infty} \ge |v_\ell - v_{\ell_+}| \ge |v_i - v_{\ell_+}| = L $.

	\noindent
	In view of all these cases, we have $	\|  \B P \B v -\B  v\|_{\infty} \ge L = 	\|  \tdP \B v - \B v\|_{\infty}  $.
\end{proof}

\subsection{Proof of Lemma~\ref{lemma: one-step}}\label{subsection: proof of one-step}

Without loss of generality we assume $\B P^* = \B I_n$, (otherwise, we work with $\B P (\B P^*)^{-1} $ in place of $\B P $ and $ \tdP (\B P^*)^{-1}$ in place of $\tdP$).
For any $k\in [n]$, let $k_+:= \pi_{\B P}(k)$. Let $ i $ be the index such that $( y_{i_+} - y_i)^2 = \| \B P \B y- \B y\|^2_{\infty} $.
With out loss of generality, we can assume $ y_{i_+} > y_i $. Denote $i_0 = i$ and $i_1 = i_+$. 
By the structure of a permutation, there exists a cycle that
\begin{equation}\label{cycle}
\setlength{\abovedisplayskip}{3pt}
\setlength{\belowdisplayskip}{3pt}
i_0 \mathop{\longrightarrow}\limits^{P} i_1 \mathop{\longrightarrow}\limits^{P}  \cdots
\mathop{\longrightarrow}\limits^{P} 
i_t \mathop{\longrightarrow}\limits^{P}  \cdots
\mathop{\longrightarrow}\limits^{P} i_S = i_0 
\end{equation}
where $q_1 \mathop{\longrightarrow}\limits^{P} q_2$ means $q_2 = \pi_{\B P}(q_1)$. 
By moving from $y_i $ to $y_{i_+}$, 
we ``upcross" the value $ \frac{y_i+y_{i_+}}{2}  $. Since the cycle \eqref{cycle} finally returns to $i_0$, there exists one step where we ``downcross" the value $ \frac{y_i+y_{i_+}}{2}  $.
In other words,
there exists $j\in[n]$ with $(j, j_+) \neq (i, i_+)$ such that $ y_{j_+} < y_j $ and $ \frac{y_i+y_{i_+}}{2} \in [y_{j_+}, y_j] $. 
Define $\tdP$ as follows:
$$
\pi_{\tdP} (i) = j_+,   ~~~	\pi_{\tdP} (j) = i_+, ~~~ \pi_{\tdP}(k) = \pi_{\B P}(k) ~~ \forall k\neq i,j \ .
$$
We note that $ \dist (\B P, \tdP) = 2 $ and 
$\supp(\tdP ) \subseteq \supp( \B P ) $.
Since 
$$
y_{i_+} - y_i = \| \B P\B y-\B y\|_{\infty} \ge y_{j} - y_{j_+}
$$
there are 3 cases depending upon the relative ordering $ y_i, y_{i_+}, y_j, y_{j_+}$, as considered below. 

\vspace{0.3cm}
\noindent
\textbf{Case 1}: $ y_j \ge y_{i_+} \ge y_{j_+} \ge y_i $. 
In this case, let $  a = y_j - y_{i_+} $, $b = y_{i_+} - y_{j_+}$ and $c = y_{j_+} - y_i   $. Then $a,b,c \ge 0$, and
\begin{eqnarray}
&&\| \B P \B y - \B y\|^2  - \| \tdP \B y - \B y \|^2 \nonumber\\
&=&
(y_i - y_{i_+})^2 + (y_j - y_{j_+})^2 - (y_i - y_{j_+})^2 - (y_j - y_{i_+})^2 \nonumber\\
&=&
(b+c)^2 + (a+b)^2 - c^2 - a^2 \nonumber\\
&=&
2b^2 + 2ab + 2 bc \ . \nonumber
\end{eqnarray}
Since $ \frac{y_i+y_{i_+}}{2} \in [y_{j_+}, y_j] $, we have 
$b = y_{i_+} - y_{j_+} \ge y_{i_+} - \frac{y_i+y_{i_+}}{2} = \frac{y_{i_+}- y_i}{2} $,
and hence 
$$
\| \B P \B y - \B y\|^2  - \| \tdP \B y - \B y \|^2 \ge 2b^2 \ge {(y_{i_+}- y_i)^2}/{2} = (1/2) \| \B P \B y - \B y\|^2_{\infty}  \ . 
$$

\vspace{0.1cm}
\noindent
\textbf{Case 2}: $ y_{i_+} \ge y_{j} \ge y_{i} \ge y_{j_+} $. 
In this case, let $  a = y_{i_+} - y_{j}  $, $b = y_{j} - y_{i}$ and $c = y_{i} - y_{j_+}   $. Then $a,b,c \ge 0$, and
\begin{eqnarray}
&&\| \B P \B y - \B y\|^2  - \| \tdP  \B y - \B y \|^2 \nonumber\\
&=&
(y_i - y_{i_+})^2 + (y_j - y_{j_+})^2 - (y_i - y_{j_+})^2 - (y_j - y_{i_+})^2 \nonumber\\
&=&
(a+b)^2 + (b+c)^2 - c^2 - a^2 \nonumber\\
&=&
2b^2 + 2ab + 2 bc \ . \nonumber
\end{eqnarray}
Since $ \frac{y_i+y_{i_+}}{2} \in [y_{j_+}, y_j] $, we have $ b = y_{j} - y_{i} \ge \frac{y_i+y_{i_+}}{2} - y_i = \frac{y_{i_+}- y_i}{2}$, 
and hence 
$$
\| \B P  \B y - \B y\|^2  - \| \tdP \B y - \B y \|^2 \ge 2b^2 \ge {(y_{i_+}- y_i)^2}/{2} = (1/2) \| \B P \B y - \B y\|^2_{\infty} \ . 
$$

\vspace{0.1cm}
\noindent
\textbf{Case 3}: $ y_{i_+} \ge y_{j} \ge y_{j_+} \ge y_{i} $. 
In this case, let $  a = y_{i_+} - y_{j}  $, $b = y_{j} - y_{j_+}$ and $c = y_{j_+} - y_{i}   $. Then $a,b,c \ge 0$, and
\begin{eqnarray}
&&\| \B P \B y - \B y\|^2  - \| \tdP \B y - \B y \|^2 \nonumber\\
&=&
(y_i - y_{i_+})^2 + (y_j - y_{j_+})^2 - (y_i - y_{j_+})^2 - (y_j - y_{i_+})^2 \nonumber\\
&=&
(a+b+c)^2 + b^2 - c^2 - a^2 \nonumber\\
&=&
2b^2 + 2ab + 2 bc + 2ac \ . \nonumber
\end{eqnarray}
Note that $ \| \B P \B y- \B y \|_{\infty}^2 = (y_i-y_{i_+})^2 = (a+b+c)^2 $. Because $ \frac{y_i+y_{i_+}}{2} \in [y_{j_+}, y_j] $, we have 
\begin{eqnarray}
a = y_{i_+} - y_j \le y_j - y_i = b+c, \quad \text{and}~~
c = y_{j_+} - y_i \le y_{i_+} - y_{j_+} = a+b\ , \nonumber
\end{eqnarray}
which implies 
$ a \le (a+b+c)/2 $ and $ c\le (a+b+c)/2  $. So we have
\begin{eqnarray}
\| \B P \B y - \B y\|^2  - \| \tdP \B y - \B y \|^2 \ge w \| \B P \B y- \B y \|_{\infty}^2 \nonumber
\end{eqnarray}
where 
$$
\setlength{\abovedisplayskip}{3pt}
w :=  \min \Big\{ \frac{2b^2 + 2ab + 2 bc + 2ac}{(a+b+c)^2}~: ~ a,b,c\ge 0; ~ a,c\le(a+b+c)/2 \Big\} .
$$
This is equivalent to 
\begin{eqnarray}
w &=& \min \Big\{ 
2b^2 + 2ab + 2 bc + 2ac~ : ~ a,b,c\ge 0; ~ a,c\le 1/2; ~ a+b+c = 1
\Big\} \nonumber\\
&=&
\min \Big\{ 
2b +  2ac~ : ~ a,b,c\ge 0; ~ a,c\le 1/2; ~ a+b+c = 1
\Big\} \nonumber\\
&=&
\min \Big\{ 
2(1-a-c) +  2ac~ : ~ a,c\ge 0; ~ a,c\le 1/2
\Big\} \nonumber\\
&=&
\min \Big\{ 
2(1-a)(1-c)~ : ~ a,c\ge 0; ~ a,c\le 1/2
\Big\} \nonumber\\
&=& 1/2 \ . \nonumber
\end{eqnarray}
Combining Cases 1, 2 and 3, completes the proof of  \eqref{one-step ineq1}.

For the proof of \eqref{one-step ineq2}, 
note that if $ \| \B P\B y -\B  P^* \B y \|_{0} \le m$, then $  \| \B P \B y - \B P^* \B y \|^2  \le m \| \B P \B y - \B P^* \B y \|^2_{\infty}$. Using \eqref{one-step ineq1} we have: 
$$
\| \B P \B y - \B P^* \B y\|^2  - \| \tdP \B y - \B P^* \B y \|^2 \ge (1/2)  \| \B P \B y -\B P^* \B y\|^2_{\infty} \ge (1/(2m))  \| \B P \B y -\B P^* \B y\|^2 ,
$$
which completes the proof of \eqref{one-step ineq2}.

\subsection{Proof of Lemma \ref{lemma: large decrease in past iterations}}\label{section: proof lemma large decrease in past iterations}

	To prove Lemma~\ref{lemma: large decrease in past iterations}, we first prove the following proposition:
	
	\begin{proposition}\label{prop-new}
		Under the assumptions of Lemma~\ref{lemma: large decrease in past iterations}, it holds 
		$\| \B P^{(t)} \B y - \B P^* \B y \|_{\infty} \ge L $ for all $0 \le t \le k$. 
	\end{proposition}
\begin{proof}
	As $\| \B P^{(k)} \B y - \B P^* \B y \|_{\infty} \ge L $,  	let $i\in [n]$ be the index such that $ | \B e_i^\T (\B P^{(k)} \B y - \B P^* \B y) | \ge L $. 
		We can assume that there exists $j \le k-1$ such that 
		\begin{eqnarray}\label{ass-exist-j}
		\B e_i^\T \B P^{(j)} \neq \B e_i^\T \B P^{(k)} ~~~ {\rm but } ~~~ \B e_i^\T \B P^{(t)} = \B e_i^\T \B P^{(k)} ~~ \forall ~ j+1 \le t \le k    \ .  
		\end{eqnarray}
		since otherwise $\B e_i^\T \B P^{(t)} = \B e_i^\T \B P^{(k)} $ for all $0\le t \le k$ and hence $\| \B P^{(t)} \B y - \B P^* \B y \|_{\infty} \ge | \B e_i^\T (\B P^{(t)} \B y - \B P^* \B y) | \ge L$ for all $0\le t \le k$, i.e., the conclusion of Proposition~\ref{prop-new} holds true. 
		
		Below we prove Proposition~\ref{prop-new} under the assumption \eqref{ass-exist-j}. 
	By \eqref{ass-exist-j} we know that $ \B e_i^\T \B P^{(j)} \neq \B e_i^\T \B P^{(j+1)} $. For any $t\ge 0$ let $\B Q^{(t)} := \B P^{(t)} (\B P^*)^{-1}$. Then we have
	$$
	\B e_i^\T \B Q^{(j+1)} = \B e_i^\T \B P^{(j+1)} (\B P^*)^{-1} \neq \B e_i^\T \B P^{(j)} (\B P^*)^{-1} = \B e_i^\T \B Q^{(j)}  \ .
	$$
	Since $ \dist (\B Q^{(j)}, \B Q^{(j+1)}) = \dist (\B P^{(j)}, \B P^{(j+1)}) =2$, there must exist an index $\ell \in [n]$ such that 
	$$
	\pi_{\B Q^{(j+1)}} (i) = \pi_{\B Q^{(j)}} (\ell )  ,~~ 	\pi_{\B Q^{(j+1)}} (\ell) = \pi_{\B Q^{(j)}} (i)  ,~~
	\pi_{\B Q^{(j+1)}} (w) = \pi_{\B Q^{(j)}}(w) ~~ \forall w\neq i,\ell \ . 
	$$
	In the following, denote $i_+:= \pi_{\B Q^{(j)}} (i )$ and $\ell_+:= \pi_{\B Q^{(j)}} (\ell )$ and let $\B y^*:= \B P^* \B  y$. Then we have
	\begin{eqnarray}\label{eq--1}
	&&\| \B P^{(j+1)} \B y - \B P^* \B y \|^2  -  \| \B P^{(j)} \B y - \B P^* \B y \|^2 \nonumber\\
	&=&
	\| (\B Q^{(j+1)} - \B I_n) \B y^* \|^2 - \| (\B Q^{(j)} - \B  I_n) \B y^* \|^2 \nonumber\\
	&=&
	(y_i^* - y_{\ell_+}^*)^2 + (y_{\ell}^* - y_{i_+}^*)^2 - 	(y_i^* - y_{i_+}^*)^2 - 	(y_{\ell}^* - y_{\ell_+}^*)^2  \ .
	\end{eqnarray}
	Since $\ell_+ =  \pi_{\B Q^{(j)}} (\ell ) =  \pi_{\B Q^{(j+1)}} (i)$, we have
	\begin{equation}\label{eq-1}
	(y_i^* - y_{\ell_+}^*)^2 = ( \B e_i^\T ( \B P^* \B y - \B Q^{(j+1)} \B P^* \B y) )^2 =  ( \B e_i^\T ( \B P^* \B y - \B P^{(j+1)} \B y) )^2  \ .
	\end{equation}
	By the definition of $j$ and equality \eqref{eq-1}, we have
	\begin{equation}\label{ineq--3}
	(y_i^* - y_{\ell_+}^*)^2 =  ( \B e_i^\T ( \B P^* \B y - \B P^{(k)} \B y) )^2 \ge L^2  \ .
	\end{equation}
	
	Since $ \| \tdH \B P^{(j+1)} \B y \| \le \| \tdH \B P^{(j)} \B y \|  $, using Lemma \ref{lemma: step-bound} with $\tdP = \B P^{(j+1)}$ and $\B P = \B P^{(j)}$, we have 
	$$
	\| \B P^{(j+1)} \B y - \B P^* \B y \|^2 - \| \B P^{(j)} \B y - \B P^* \B y \|^2 < 
	\| \tdH \B P^{(j+1)} \B y \|^2 - \| \tdH \B P^{(j)} \B y \|^2  + L^2/5 
	\le  L^2/5.
	$$
	This leads to 	
	$$
	(y_i^* - y_{\ell_+}^*)^2 + (y_{\ell}^* - y_{i_+}^*)^2 - 	(y_i^* - y_{i_+}^*)^2 - 	(y_{\ell}^* - y_{\ell_+}^*)^2  < L^2/5
	$$
	which when combined with \eqref{ineq--3} leads to: 
	$$
	(y_i^* - y_{i_+}^*)^2 + 	(y_{\ell}^* - y_{\ell_+}^*)^2 > (y_i^* - y_{\ell_+}^*)^2 + (y_{\ell}^* - y_{i_+}^*)^2 - (1/5)L^2 \ge (4/5)L^2.
	$$
	As a result, we have 
	\begin{equation}\label{ineq-8}
	\| \B P^{(j)} \B y - \B P^* \B y \|_{\infty}^2  \ge \max \{(y_i^* - y_{i_+}^*)^2, (y_{\ell}^* - y_{\ell_+}^*)^2\} > (2/5)L^2
	\end{equation}
	By Lemma \ref{lemma: one-step}, there exists $\tdP^{(j)} \in \Pi_n$ such that $\dist(\tdP^{(j)}, \B P^{(j)}) \le 2$, $ \supp(\tdP^{(j)} (\B P^*)^{-1} )  \subseteq \supp(\B P^{(j)} (\B P^*)^{-1} )$ and 
	\begin{equation}\label{ineq-6}
	\| \tdP^{(j)} \B y - \B P^*  \B y \|^2 - \|  \B P^{(j)} \B y - \B P^* \B y \|^2 \le - (1/2) \| \B P^{(j)} \B y - \B P^* \B y \|_{\infty}^2 \ . 
	\end{equation}
	We now make use of the following claim:
	\begin{equation}\label{claim3}
	{\bf Claim.} ~~~~~~~~~~~~~~~ \tdP^{(j)} \in \cN_R( \B I_n) .
	\end{equation}
	\begin{itemize}
		\item[] 
		\emph{Proof of Claim~\eqref{claim3}:}	Note that if $j\le R/2 - 1$, then $ \dist(\B I_n, \tdP^{(j)} ) \le \dist(  \B I_n, \B P^{(j)} ) + \dist( \B P^{(j)}, \tdP^{(j)} ) \le R$. Otherwise, from the statement of Lemma \ref{lemma: large decrease in past iterations}, we know that $\supp(\B P^*) \subseteq \supp(\B P^{(j)}) $. Using Lemma~\ref{lemma: supp-inclusion} with $\B P = \B P^{(j)}$, $\B Q = \B P^*$ and $\tdP = \tdP^{(j)}$ we have $\supp(\tdP^{(j)}) \subseteq \supp(\B P^{(j)})$. Since $\B P^{(j)} \in \cN_R(\B I_n)$, we also have $\supp(\tdP^{(j)}) \in \cN_R(\B I_n)$. The proof of Claim \eqref{claim3} is complete. 
	\end{itemize}
	Because $\tdP^{(j)} \in \cN_R(\B I_n)$, by the update step in the local search algorithm, we know $ \| \tdH \B P^{(j+1)} \B y \| \le  \| \tdH \tdP^{(j)} \B y \|  $. 
	Using Lemma \ref{lemma: step-bound} again with $\tdP = \B P^{(j+1)}$ and $\B P = \tdP^{(j)}$, we have 
	\begin{equation}\label{ineq-7}
	\begin{aligned}
	&\| \B P^{(j+1)} \B y -\B P^* \B y \|^2 - 	\| \tdP^{(j)} \B y - \B P^* \B y \|^2 \\
	\le & \| \tdH \B P^{(j+1)} \B y \|^2 -  \| \tdH \tdP^{(j)} \B y \|^2 + L^2/5
	\le  L^2/5
	\end{aligned}
	\end{equation}
	Combining \eqref{ineq-6}, \eqref{ineq-7} and \eqref{ineq-8} we have 
	$$
	\| \B P^{(j+1)} \B y - \B P^* \B y \|^2 - 	\|  \B P^{(j)} \B y - \B P^* \B y \|^2 \le
	- (1/2)\| \B P^{(j)} \B y - \B P^* \B y \|_{\infty}^2  
	+ L^2/5 <0 \ 
	$$
	which is equivalent to 
	$$
	\| \B Q^{(j+1)} \B y^* - \B y^* \|^2 - \| \B Q^{(j)} \B y^* - \B y^* \|^2 <0 \ .
	$$
	We now use Lemma \ref{lemma: Non-expansiveness of infinity norm} with $\tdP = \B Q^{(j+1)}$, $\B P = \B Q^{(j)}$ and $\B v = \B y^*$, to obtain
	$$
	\| \B P^{(j)} \B y - \B P^* \B y \|_{\infty} = 
	\| \B Q^{(j)} \B y^* - \B y^* \|_{\infty} \ge \| \B Q^{(j+1)} \B y^* - \B y^* \|_{\infty} = \| \B P^{(j+1)} \B y - \B P^* \B y \|_{\infty} \ge L \ .
	$$
	
	By the arguments above we have proved that $\| \B P^{(j)} \B y - \B P^* \B y \|_{\infty} \ge L $. Recall that we also have $ \| \B P^{(t)} \B y - \B P^* \B y \|_{\infty} \ge L$ for all $ j+1 \le t\le k $, so we know that $ 	\| \B P^{(t)} \B y - \B P^* \B y \|_{\infty} \ge L $ for all $j\le t \le k$. We can just replace $k$ by $j$ and repeat the arguments above to obtain $ 	\| \B P^{(t)} \B y - \B P^* \B y \|_{\infty} \ge L $ for all $ 0 \le t\le k$. This completes the proof of Proposition~\ref{prop-new}. 
\end{proof}

With Proposition~\ref{prop-new} at hand, we are ready to wrap up the proof of Lemma~\ref{lemma: large decrease in past iterations}. 
For each $t\le k-1$, By Lemma \ref{lemma: one-step}, 
there exists $\tdP^{(t)} \in \Pi_n$ such that 
$	\dist(\tdP^{(t)}, \B P^{(t)}) \le 2  $, $\supp(\tdP^{(t)} (\B P^*)^{-1} )  \subseteq \supp(\B P^{(t)} (\B P^*)^{-1} )$ and 
\begin{eqnarray}\label{ineq-p1}
\| \tdP^{(t)} \B y - \B P^* \B y \|^2 -  \|  \B P^{(t)} \B y - \B P^* \B y \|^2 \le - (1/2) \| \B P^{(t)} \B y - \B P^* \B y \|_{\infty}^2 \le 
- L^2/2 \ 
\end{eqnarray}
	where the second inequality is by Proposition \ref{prop-new}. 
With the same arguments as in the proof of Claim \eqref{claim3}, we have $\tdP^{(t)} \in \cN_R(\B I_n)$, hence $\| \tdH \B P^{(t+1)} \B y \|^2 \le \| \tdH \tdP^{(t)} \B y \|^2$. 
Using Lemma \ref{lemma: step-bound} again with $\B P = \tdP^{(t)}$ and $\tdP =\B  P^{(t+1)}$, we have 
\begin{equation}\label{ineq-p2}
\begin{aligned}
& 	\| \B P^{(t+1)} \B y- \B P^* \B y \|^2 - 	\| \tdP^{(t)} \B y- \B P^* \B y \|^2 \\
\le& \| \tdH \B P^{(t+1)} \B y \|^2 - \| \tdH \tdP^{(t)} \B y \|^2  + L^2/5
\le  L^2/5  \ .
\end{aligned}
\end{equation}
Combining \eqref{ineq-p1} and \eqref{ineq-p2} we have 
$$
\| \B P^{(t+1)} \B y- \B P^* \B y \|^2 - \|  \B P^{(t)} \B y - \B P^*  \B y \|^2  \le L^2/5 - L^2/2 < -L^2/5
$$
which completes the proof of the Lemma~\ref{lemma: large decrease in past iterations}.

\subsection{Proof of Lemma \ref{lemma: rip}}\label{subsection: proof of rip lemma}

We will prove that 
for any $ \B u \in  \cB_m$ (cf definition~\eqref{defn-Bm}),
\begin{eqnarray}\label{app:lemma1-to-prove}
\| \B H \B u \|^2 = \| \B X (\B X^\T \B X)^{-1} \B X^\T  \B u\|^2 \le \dt_{n,m} \| \B u\|^2 \ .
\end{eqnarray}
Take $t_n := \sqrt{3b^2 \log(2d/\tau) / (n\itl \B \Sigma\itl_2)}$. By assumption in the statement of Lemma \ref{lemma: rip} 
we have $t_n \itl \B\Sigma \itl_2 \le \ga/2$. From Lemma~\ref{lemma: sample covariance estimation} and some simple algebra we have:
\begin{equation}\label{ineq: covariance bound}
\itl  (1/n) \B X^\T \B X - \B \Sigma \itl_2 \le t_n \itl \B \Sigma\itl_2 
\end{equation}
with probability at least $1-\tau$. 
By Weyl's inequality, $|\lam_{\min} (\B X^\T \B X)/n- \lam_{\min} (\B \Sigma)|$ is bounded by the left hand side of \eqref{ineq: covariance bound}. So we have 
\begin{equation*}
\lam_{\min} (\B X^\T \B X)/n \ge 
\lam_{\min} (\B \Sigma)- t_n \itl \B\Sigma \itl_2 
\ge  \ga -  \ga/2 =  \gamma/2 \  \end{equation*}
where, we use $t_{n}\itl \B\Sigma \itl_2 \leq \ga/2$.
Hence we have $\lam_{\max}((\B X^\T \B X)^{-1}) \le {2}/{(n \ga)}  $ and
\begin{eqnarray}\label{ineq-XXX}
\itl \B X (\B X^\T \B X)^{-1}\itl_2 = \sqrt{\lam_{\max}((\B X^\T \B X)^{-1})} \le \sqrt{2/(n\ga)} \ .
\end{eqnarray} 
Let $ \cB_m(1) := \{\B u\in \cB_m: ~ \|\B u\|\le1 \}$, and let $\B u^1,...,\B u^M$ be an $(\sqrt{\dt_{n,m}}/2)$-net of $\cB_m(1)  $, that is, for any $\B u\in \cB_m(1)$, there exists some $\B u^j$ such that $ \|\B u^j - \B u\| \le \sqrt{\dt_{n,m}}/2 $. 
Since the $(\sqrt{\dt_{n,m}}/2) $-covering number of $\cB_m(1)$ is bounded by 
$({6}/{\sqrt{\dt_{n,m}}})^m  {n \choose m}$,
we can take 
$$
M \le ({6}/{\sqrt{\dt_{n,m}}})^m {n \choose m} \le   (3n)^m n^m = (3n^2)^m
$$
where the second inequality is from our assumption that $\sqrt{\dt_{n,m}}\ge 2/n$. 
By Hoeffding inequality, for each fixed $j\in [M]$, and for all $ k\in [d] $, we have
\begin{eqnarray}
\Pb \Big(     \frac{1}{\sqrt{n}} \lt| \B e_k^\T \B X^\T \B u^j  \rt| > t\Big)  
\le 2 \exp \Big( - \frac{nt^2}{2 \|\B u^j\|^2 V^2} \Big) \ . \nonumber
\end{eqnarray}
Using a union bound for $k\in [d]$ to the inequality above, we have that
for any $\rho>0$,  
with probability at least $  1- \rho $, the following inequality holds for all $k\in [d]$
\begin{eqnarray}
\lt| \B e_k^\T \B X^\T \B u^j  \rt| /\sqrt{n} \le \sqrt{{2}  \log(2d/\rho) /n} V \| \B u^j\|\le V\sqrt{{2} \log(2d/\rho) /n} \ , \nonumber
\end{eqnarray}
where the second inequality is because each $\B u^j \in \cB_m(1)$. As a result, 
$$
\setlength{\abovedisplayskip}{5pt}
\setlength{\belowdisplayskip}{5pt}
\frac{1}{\sqrt{n}} \| \B X^\T \B u^j \| = \Big(\sum_{k=1}^d   \lt(  | \B e_k^\T \B X^\T \B u^j  |/\sqrt{n} \rt)^2     \Big)^{1/2}  \le V\sqrt{{2d} \log(2d/\rho) /n} \ .
$$
Take $\rho = \tau/ M$, then by the union bound, with probability at least $1-\tau$, it holds
\begin{eqnarray}\label{ineq: nets}
\| \B X^\T \B u^j \|/\sqrt{n}  \le V\sqrt{{2d} \log(2dM/\tau)/n } ~ ~~~ \forall ~ j\in [M] \ .
\end{eqnarray}
Combining~\eqref{ineq: nets} with \eqref{ineq-XXX}, we have that for all $j\in [M]$,
\begin{equation}\label{append:lemma1-proof11} 
\setlength{\belowdisplayskip}{3pt}
\begin{aligned}
\| \B X(\B X^\T \B X)^{-1} \B X^{\T} \B u^j \| \le& \itl \B X(\B X^\T \B X)^{-1}\itl_2\cdot  \| \B X^\T \B u^j \|\\
\le& 2V \sqrt{({d}/{n\ga }) \log(2dM/\tau) }   \ . 
\end{aligned}
\end{equation}
Recall that $ M\le (3n^2)^m $, so we have
\begin{eqnarray}
2V \sqrt{({d}/{n\ga }) \log(2dM/\tau) } \le 2V \Big( \frac{d}{n\ga}\log(2d/\tau)   + \frac{dm}{n\ga} \log(3n^2) \Big)^{1/2} = \frac{\sqrt{\dt_{n,m}}}{2} \nonumber
\end{eqnarray}
where the last equality follows the definition of $\dt_{n,m}$. 
Using the above bound in~\eqref{append:lemma1-proof11}, we have
$$
\setlength{\abovedisplayskip}{3pt}
\| \B X(\B X^\T \B X)^{-1} \B X^{\T} \B u^j \| \le \sqrt{\dt_{n,m}}/2.
$$
For any $\B u\in \cB_m(1)$, there exists some $j\in [M]$ such that $\| \B u- \B u^j\| \le \sqrt{\dt_{n,m}}/2$, hence 
\begin{eqnarray}\label{ineq-final}
\| \B X(\B X^\T \B X)^{-1} \B X^{\T} \B u \| &\le& \| \B X(\B X^\T \B X)^{-1} \B X^{\T} \B u^j \|  + \| \B X(\B X^\T \B X)^{-1} \B X^{\T} (\B u-\B u^j) \|  \nonumber\\
&\le&
\sqrt{\dt_{n,m}}/2 + \| \B u- \B u^j\|_2 
\le
\sqrt{\dt_{n,m}} \ .
\end{eqnarray}
Since both \eqref{ineq: covariance bound} and \eqref{ineq: nets} have failure probability of at most $\tau$, we know that 
\eqref{ineq-final} holds
with probability at least $1-2\tau$. This proves the conclusion for all $\B u\in \cB_m(1)$. For a general $\B u\in \cB_m$, $ \B u/ \| \B u\| \in  \cB_m(1)$, hence we have
$$
\setlength{\abovedisplayskip}{5pt}
\setlength{\belowdisplayskip}{5pt}
\| \B H \B u\|  = \| \B X(\B X^\T \B X)^{-1}\B X^{\T} \B u \| \le \sqrt{\dt_{n,m}}  \|\B u\|
$$
which is equivalent to what we had set out to prove~\eqref{app:lemma1-to-prove}.

	\subsection{Proof of Lemma~\ref{lemma: noise}}\label{section: proof of lemma noise}
	
	Since $\ep_i\sim \subG(\sigma^2)$, by Chernoff inequality, for all $t>0$, we have 
	$$
	\Pb( |\ep_i| >t  ) \le 2 \exp (-t^2/(2\sigma^2)) \ .
	$$
	As a result, for all $t>0$, 
	$\Pb (\| \B \ep\|_{\infty}>t) \le \sum_{i=1}^n \Pb( |\ep_i| >t  ) \le 2n \exp (-t^2/(2\sigma^2)) $.
	Therefore with probability at least $1-\tau/3$ we have  $\| \B \ep \|_{\infty} \le \sigma \sqrt{2\log(6n/\tau)}$. 
	
	For any $i\in [n]$, 
	since $\| \B e_i^\T \tdH \|^2 \le 1   $, so it is easy to check that $\B e_i^\T \tdH \B\ep $ is also sub-Gaussian with variance proxy $\sigma^2$. Similar to the arguments in the last paragraph, with probability at least $1- \tau/3$, we have $\| \tdH \B \ep \|_{\infty} \le \sigma \sqrt{2\log(6n/\tau)}$. 
	So we have proved that with probability at least $1-2\tau/3$, the inequalities in (a) hold. 
	
	Let $\B X = \B{\bar U} \B D \B{\bar V}^\T$ be the SVD of $\B X$ with $\B{\bar U} \in \R^{n\times d}$ satisfying $\B{\bar U}^\T \B{\bar U} = \B I_d$; $\B{\bar V} \in \R^{d\times d}$ satisfying $ \B{\bar V}^\T \B{\bar V} = \B I_d$; and $\B D \in \R^{d\times d}$ being diagonal. Then we have 
	\begin{equation}
	\|  (\B X^\T \B X)^{-1}\B X^{\T}  \B \ep  \|  =  
	\| \B{\bar V}  \B{D}^{-1} \B{\bar U}^\T  \B \ep \| 
	=  \|   \B{D}^{-1} \B{\bar U}^\T  \B \ep \| 
	\le 
	\itl   \B{D}^{-1}   \itl_2  \sqrt{d} \| \B{\bar U}^\T \B \ep \|_{\infty} 
	\end{equation}
	and 
	\begin{equation}
	\| \B H \B \ep \|  = \|  \B{\bar U} \B{\bar U}^\T \B \ep \| = \| \B{\bar U}^\T \B \ep  \| \le \sqrt{d} \| \B{\bar U}^\T \B \ep \|_{\infty} 
	\end{equation}
	Note that for any $j\in [d]$, one can verify that $\B e_j^\T \B{\bar U}^\T \B\ep$ is sub-Gaussian with variance proxy $ \sigma^2$. Hence, for any $t>0$, we have 
	$$
	\Pb(\|\B{\bar U}^\T \B \ep \|_{\infty} >t ) \le 2 d \exp(-t^2/(2\sigma^2))  \ . 
	$$
	As a result, with probability at least $1-\tau/3$ we have 
	$\| \B{\bar U}^\T \B \ep \|_{\infty} \le \sigma \sqrt{2\log(6d/\tau)}  $, and hence 
	\begin{equation}
	\|  (\B X^\T \B X)^{-1}\B X^{\T}  \B \ep  \|  \le  
	\itl   \B{D}^{-1}   \itl_2  \sigma \sqrt{2d\log(6d/\tau)}  \ ,
	\end{equation}
	and 
	\begin{equation}
	\| \B H \B\ep \| \le \sigma \sqrt{2d\log(6d/\tau)} \ .
	\end{equation}
	Note that $	\itl   \B{D}^{-1}   \itl_2 = \frac{1}{\sqrt{n}} \lam_{\min}^{-1/2} (\frac{1}{n} \B X^\T \B X ) $, so 
	this completes the proofs of (b) and (c).

	To prove \eqref{conclusion2-noise}, using Bernstein inequality we have
	$$
	\Pb (| \| \B \ep\|^2 /n - \wtd \sigma^2| > t\sigma^2 ) \le 2\exp(-Cn(t^2 \wedge t))
	$$
	for a universal constant $C$. 
	Taking $t = \sqrt{\log(4/\tau)/(Cn)}$ in the inequality above, and note that $t\le 1$ (because of the assumption $\sqrt{\log(4/\tau)/(Cn)} + 2d \log(4d/\tau)/n \le 1/4 $), we have 
	\begin{equation}
	\Pb (| \| \B \ep\|^2 /n - \wtd \sigma^2| > \sqrt{\log(4/\tau)/(Cn)}\sigma^2 ) \le 2\exp(-Cn t^2 ) = \tau/2.
	\end{equation}
	As a result,
	with probability at least $1-\tau/2$
	\begin{eqnarray}
	\| \B \ep \|^2 \ge  n \wtd\sigma^2 - n\sigma^2 \sqrt{\log(4/\tau)/(Cn)} \ge 
	(3/4-\sqrt{\log(4/\tau)/(Cn)} \ ) n\sigma^2 \  \nonumber
	\end{eqnarray}
	where the second inequality makes use of the assumption that $\wtd \sigma^2 \ge (3/4) \sigma^2$. 
	On the other hand we know that with probability at least $1-\tau/2$ it holds $\| \B H \B \ep \|\le  \sigma \sqrt{2d\log(4d/\tau)}$. As a result, with probability at least $1-\tau$ we have
	\begin{eqnarray}
	\|\tdH \B \ep \|^2 = \|\B \ep\|^2 - \|\B H \B \ep \|^2 \ge (3/4-\sqrt{\log(4/\tau)/(Cn)}  -   2d \log(4d/\tau)/n) n\sigma^2 \ge n\sigma^2/2 \ , \nonumber
	\end{eqnarray}
	where the last inequality uses the assumption $\sqrt{\log(4/\tau)/(Cn)}  +   2d \log(4d/\tau)/n \le 1/4$.

\subsection{Proof of Claim \eqref{claim2}}\label{subsection: proof of claims}

To prove Claim \eqref{claim2}, 
we just need to prove that $\tdP^{(k)} \in \cN_R(\B I_n)$, i.e. $\dist(\tdP^{(k)} ,\B  I_n) \le R$. If $k\le R/2 - 1$, because $ \dist(\B P^{(t+1)}, \B P^{(t)}) \le 2 $ for all $t\ge 0$ and $\B P^{(0)} =\B  I_n$, 
we have $ \dist(\B P^{(k)}, \B I_n) \le 2k \le R-2  $. Hence 
$$
\setlength{\abovedisplayskip}{3pt}
\setlength{\belowdisplayskip}{3pt}
\dist(\tdP^{(k)}, \B I_n) \le \dist(\tdP^{(k)}, \B P^{(k)}) + \dist( \B P^{(k)}, \B I_n)  \le  R \ .
$$
Otherwise, by Proposition~\ref{prop: support-inclusion}, it holds $ \supp( \B P^* ) \subseteq \supp(\B P^{(k)})$. 
Then from Lemma \ref{lemma: supp-inclusion}, we have $ \supp(\tdP^{(k)}) \subseteq   \supp(\B P^{(k)})$, therefore $\tdP^{(k)} \in \cN_R( \B I_n)$.

\subsection{Proof of Claim \eqref{claim-bound}}\label{subsection: proof of claim-bound}
Note that 
\begin{equation}\label{ineq--0}
|\la  \tdH \B z ,  \tdH \B \ep \ra | \le 	|\la   \B H \B z ,  \B H \B \ep \ra |  + 
|\la   \B z ,   \B \ep \ra |   \ . 
\end{equation}
Let $w = 33$. 
Since only two coordinates of $z$ are non-zero, we have 
\begin{equation}\label{ineq--1}
|\la   \B z ,   \B \ep \ra | \le \| \B z\| \sqrt{2}\bar \sigma  \le \eta\| \B z\|^2/w + w\bar \sigma^2 /(2\eta) \ , 
\end{equation}
where the second inequality uses Cauchy-Schwarz inequality. 
Using Assumption \ref{ass-main0} (3) and (4) we have 
$$
|\la  \B H \B z ,  \B H \B \ep \ra |  \le \| \B H \B z \| \| \B H \B \ep \| \le \sqrt{\rho_n} \|\B z\| \sqrt{d} \bar \sigma \le \rho_n \|\B z\|^2 /2 +  d\bar \sigma^2/2 \ .
$$
By Assumption \ref{ass-main0} (3), we have  $\rho_nR  \le L^2/(90U^2)  \le 1/73 \le 1/(aw)$, hence
$ \rho_n/2 \le 1/(2aRw) =   \eta/w$. As a result, 
\begin{equation}\label{ineq--2}
|\la  \B H \B z ,  \B H \B \ep \ra | \le \eta\|\B z\|^2/w +   d\bar \sigma^2/2
\ .  
\end{equation}
Combining \eqref{ineq--1}, \eqref{ineq--2} and \eqref{ineq--0} we have 
\begin{equation}\label{prev1}
|\la  \tdH \B z ,  \tdH \B \ep \ra |  \le 
2\eta \|\B z\|^2/w + w\bar \sigma^2 /(2\eta)+  d\bar \sigma^2/2 \ .
\end{equation}
From $\rho_nR  \le L^2/(90U^2)  $ we know that $\rho_n \le 0.01$. Therefore by Assumption \ref{ass-main0} (3) we have 
$ \|\tdH z\|^2 \ge (1-\rho_n) \| \B z\|^2 \ge 0.99\|\B z\|^2  $. So $ \| \B z\|^2 \le 1.02 \|\tdH \B z\|^2 $ and combines it with \eqref{prev1} we have
\begin{eqnarray}
|\la  \tdH \B z ,  \tdH \B \ep \ra | 
&\le& 
2.04 \eta  \|\tdH \B z\|^2/w+ w\bar \sigma^2 /(2\eta) +  d\bar \sigma^2/2 \nonumber\\
&\le&
4.08\eta \|\tdH \tdP^{(k)} \B y\|^2/w + 4.08\eta\|\tdH  \B P^{(k)} \B y\|^2/w +
w\bar \sigma^2 /(2\eta) +  d\bar \sigma^2/2 \ , \nonumber
\end{eqnarray}
where the second inequality uses the definition of $\B z$ and the Cauchy-Schwarz inequality. Note that $4.08/w = 4.08/33 \le 1/8$ and $w/2 = 33/2 \le 17$, so we have 
\begin{equation}\label{paart1}
|\la  \tdH \B z ,  \tdH \B \ep \ra | \le 	\eta \|\tdH \tdP^{(k)} \B y\|^2/8 + \eta\|\tdH  \B P^{(k)} \B y\|^2/8 +
17 \eta^{-1}\bar \sigma^2 +  d\bar \sigma^2/2  \ .
\end{equation}
From Assumption \ref{ass-main} (2), we have $  \bar \sigma^2 / \sigma^2 \le \min \{n/(660R^2), n/(5Rd)\}$, which implies 
\begin{equation}\label{paart2}
17\eta^{-1}\bar \sigma^2\le  (1/2)\eta n \sigma^2 \le \eta \|\tdH \B \ep \|^2 , ~~ {\rm and} ~~  d\bar \sigma^2/2 \le (1/2)\eta n \sigma^2 \le \eta\|\tdH \B \ep \|^2 \ .
\end{equation}
As a result, by \eqref{paart1} and \eqref{paart2} we have 
\begin{equation}\label{ineq-c2}
|\la  \tdH \B z ,  \tdH \B \ep \ra |  \le 
\eta \|\tdH \tdP^{(k)} \B y\|^2/8 + \eta\|\tdH  \B P^{(k)} \B y\|^2/8 + 2\eta \|\tdH \B \ep \|^2 \ . 
\end{equation}
Multiplying $2$ in both sides of \eqref{ineq-c2} we complete the proof.

	\bibliographystyle{plain}     
	\bibliography{mybib_local_search}

\end{document}